\documentclass[reqno,12pt]{amsart}
\usepackage[utf8]{inputenc}  

\usepackage{float} 
\usepackage[in]{fullpage}

\newcommand{\ds}{\displaystyle}

\newcommand{\tensor}{\otimes}
\newcommand{\leftsub}[2]{{\vphantom{#2}}_{#1}{#2}} 
\newcommand{\xycirc}[2]{\leftsub{#1}{\circ}_{#2}}

\newcommand{\op}{\mathcal}

\newcommand{\cdc}{,\dots,}

\usepackage{amsmath}%
\usepackage{amsthm}
\usepackage{amsfonts}%
\usepackage{amssymb}%
\usepackage{graphicx}
\usepackage{xy,amsthm,enumerate,xypic,array}  
\usepackage{xcolor}

\input{xy}
\xyoption{all}

\numberwithin{equation}{section}

\newtheorem{theorem}{Theorem}[section]
\theoremstyle{plain}

\newtheorem{corollary}[theorem]{Corollary}
\newtheorem{lemma}[theorem]{Lemma}
\newtheorem{proposition}[theorem]{Proposition}

\theoremstyle{definition}
\newtheorem{definition}[theorem]{Definition}
\newtheorem{example}[theorem]{Example}

\newtheorem{remark}[theorem]{Remark}

\allowdisplaybreaks[2]

\addtocounter{MaxMatrixCols}{2}

\begin{document}


\title{From configuration spaces to graph complexes via $\mathbf{FA}$-modules}
\author{Ayako Carter}
\email{ayakoc@bgsu.edu}

\author{Benjamin C.\ Ward}
\email{benward@bgsu.edu}

\maketitle

\begin{abstract} 
	Work of Gadish and Hainaut (after 
	 Petersen) models the compactly supported cohomology of a wedge of circles as a polynomial functor.  We identify the coefficients of this functor, $\Phi[n,m]$, via a cobar construction of $\mathbf{FA}$-modules.  This identification formally implies that these coefficients will arise in computations of graph homology, and we use this result to give examples of graph complexes whose homology may be embedded in  $H_c^\ast(F(S^1\vee S^1,n))$.

This includes the Payne-Willwacher marked graph complex in genus 2, for which we give a new, explicit decomposition in terms of $\mathbf{FA}$-modules.  This allows us to describe $\mathsf{gr}_{11}H_c^{\ast}(\op{M}_{2,n})$ as the cohomology of a complex of decorated trees and to show, for example, $\mathsf{gr}_{11}H_c^{n+1}(\op{M}_{2,n})=0$.

\end{abstract}


\section{Introduction}
\subsection{Context}\label{context}
Write $R_g:= S^1\vee\dots\vee S^1$ for the wedge of $g$ circles.   To begin, let's consider two statements about the space $R_g$.  The first statement (Equation $\ref{statement1}$) is quite elementary, the second (Equation $\ref{statement2}$) much less so, but they will ultimately be seen to be formal consequences of each other.

For the first statement, let $\mathsf{Inj}_{m}(n)$ denote the $\mathbb{Q}$-span of the set of injections from $\bar{m}:=\{1\cdc m\}$ to $\bar{n}:=\{1 \cdc n\}$.  Write $H^\ast(R_g)$ (resp.\  $\tilde{H}^\ast(R_g)$) for the rational cohomology (resp.\ reduced rational cohomology) of the space $R_g$.  Then there is an isomorphism of graded $S_n$-representations
\begin{equation}\label{statement1}
	H^\ast(R_g)^{\tensor n} = \bigoplus_{m}\textsf{Inj}_m(n)\tensor_{S_m} \tilde{H}^{\ast}(R_g)^{\tensor m}.
\end{equation}
Here we interpret $\textsf{Inj}_m(n)=0$ when $n<m$, so this is a finite sum.

For the second statement, let us write $F(X,n)$ for the configuration space of $n$ ordered points in a space $X$.  We consider the rational cohomology with compact support, denoted $H_c^\ast(F(X,n))$.   Work of Gadish and Hainaut \cite{GH}, building on foundational work of Petersen \cite{Pet20}, shows that when $X$ is taken to be a finite wedge of spheres, the functor $X\mapsto H_c^\ast(F(X,n))$ factors as $X \mapsto \tilde{H}^\ast(X)$ composed with a polynomial functor of degree $n$.    This implies that there exists a family of graded $S_n\times S_m$ modules $\Phi[n,m]$ for which there is an isomorphism of graded $S_n$-representations
\begin{equation}\label{statement2}
	H_c^\ast(F(R_g,n))\cong \bigoplus_m\Phi[n,m]\tensor_{S_m} \tilde{H}^\ast(R_g)^{\tensor m}.
\end{equation}
Here we interpret $\Phi[n,m]=0$ when $n<m$, so this is a finite sum.

We present Statements 1 and 2 in parallel to try to suggest a relationship between them.  But to elucidate this relationship we must move from considering a specific $n$ to considering all $n$ at once.  Observe that if $f\colon\bar{n}_1\to\bar{n}_2$ is any function, there is a corresponding linear map $f_\ast\colon\mathsf{Inj}_m(n_1)\to \mathsf{Inj}_m(n_2)$ defined by $f_\ast(\phi) := f\circ \phi$ if this composition is injective and $f_\ast(\phi)=0$ if it is not.  This assignment is functorial, and indeed $\mathsf{Inj}_m$ may be viewed as a functor from the category of finite sets.  We recall common parlance for such a functor is an $\mathbf{FA}$-module.  The left hand side of Equation $\ref{statement1}$ can also be lifted to such a functor, and the equivalence lifts to an isomorphism of $\mathbf{FA}$-modules.

The notion of an $\mathbf{FA}$-module is linear dual to the operadic notion of right modules over the commutative operad, and from this operadic perspective it's natural to take the bar/cobar complex of such a module.  
Before doing so,  let us first recall a bit of the basic structure of the category of $\mathbf{FA}$-modules, drawing form \cite{Powell, CLPW2}.  For every partition $\lambda$ of $m$ we may form an $\mathbf{FA}$-module $\mathsf{C}_\lambda:=\mathsf{Inj}_m\tensor_{S_m}V_\lambda$.  For every partition not equal to $(1^m):=(1,...,1)$, such a $\mathsf{C}_\lambda$ is, in fact, simple as an $\mathbf{FA}$-module.  The only simple $\mathbf{FA}$-modules not of this form arise as submodules $\widetilde{\mathsf{C}}_{1^{m}}\subset\mathsf{C}_{1^{m-1}}$ for $m\geq 1$ .

On the other hand, by Equation $\ref{statement2}$, determining the graded $S_n$-module $H_c^\ast(F(X,n))$ is reduced to determining the graded $S_n$-modules $\Phi[n,\lambda]:=\Phi[n,m]\tensor_{S_m}V_\lambda$ for all $\lambda$ and $m\leq n$.  We relate these objects via the non-unital variant of the cobar construction $\Omega$:

\begin{theorem}\label{thm1}  Let $\lambda$ be a partition of $m$.  If $\lambda \neq (1^m)$, there exists a simple $\mathbf{FA}$-module $\mathsf{C}_\lambda$, for which 
$$	H^\ast(\Omega(\mathsf{C}_\lambda)(n)) = \Phi[n,\lambda].$$

\end{theorem} 
If $\lambda = (1^m)$ the result still holds, except $\mathsf{C}_{1^m}$ is not simple.  The family $\lambda = (1^m)$ is the one case where this homology is well understood and can be described explicitly in terms of the family $H^\ast(\Omega(\widetilde{\mathsf{C}}_{1^m}))$; see Example $\ref{WHex}$. 

We conclude that any time we take the cobar construction of a semi-simple $\mathbf{FA}$-module, the coefficients $\Phi[n,\lambda]$ are sure to arise.  One situation where this occurs is the study of graph complexes, particularly in examples of graph complexes which arise via the Feynman transform of modular operads containing the commutative operad.

\subsection{Induced $\mathbf{FA}$-modules in Lie Graph Homology}
To identify a specific example to which the above general considerations can be applied, we recall Kontsevich's Lie graph complex \cite{KontFormal}, encoded via the Feynman transform \cite{GeK2}.  If we write $\mathsf{sLie}$ for the shifted Lie operad, applying the Feynman transform returns a sequence of chain complexes $\mathsf{FT}(\mathsf{sLie})(g,0)$ which compute the homology of the groups $Out(F_g)$, the outer automorphism group of the free group on $g$ generators.  Subsequently Conant, Kassabov and Vogtmann \cite{CKV} showed that each $\mathsf{FT}(\mathsf{sLie})(g,n)$ computes the homology of a group $\Gamma_{g,n}$ which can be described, after \cite{CHKV}, as the homotopy classes of homotopy automorphisms of a wedge of $g$ circles which fix $n$ given points.

The analysis given in \cite{CHKV} can be assembled to prove:
\begin{lemma}\label{decomplem}  For each degree $d$ the symmetric sequence $H^d(\Gamma_{2,\ast})$ forms an  $\mathbf{FA}$-module of the form $\oplus \mathsf{C}_\lambda$, whose decomposition over $\lambda$ may be explicitly presented.
\end{lemma}
See Proposition \ref{g2dec}.  
This generalizes the case of $g=1$ in which the FA-module is simple; see Lemma $\ref{g1fa}$.  It thus follows that in low genus the Feynman transform of $H_\ast(\Gamma)$ can be expressed via $\Phi[n,\lambda]$.  Combining this observation with Equation $\ref{statement2}$, we compare the homology of this Feynman transform to the configuration space of points in $R_g$.  

 To give the statement we write 
 $\mathsf{FT}_\Gamma:= \mathsf{FT}(H_\ast(\Gamma))(2,-)$.  
  We then filter this symmetric sequence by the sum of the genus weights of the vertices.  Passing to the associated graded we find:
 
\begin{theorem}\label{ftthm}  	Consider $H_c^\ast(F(S^1\vee S^1),n)$ as a $S_n\times D_{12}$-module, by inclusion of the dihedral group of order 12 in $GL_2(\mathbb{Z})$.  There exist isomorphisms of graded symmetric sequences:
	\begin{enumerate}
		\item $H^{\ast+1}(\mathsf{gr}_2(\mathsf{FT}_\Gamma))\oplus \mathsf{W}_{\pm} 
		\cong H_c^{\ast}(F(S^1\vee S^1),-)\tensor_{D_{12}} V_{120}$	
		\item $H^{\ast+2}(\mathsf{gr}_1(\mathsf{FT}_\Gamma))\oplus \mathsf{W}_{\pm} \cong  H_c^\ast(F(S^1\vee S^1),-)\tensor_{D_{12}} (\chi_s \oplus V_{120} ) $	
		\item  $H^{\ast+3}(\mathsf{gr}_0(\mathsf{FT}_\Gamma))\cong H_c^\ast(F(S^1\vee S^1),-)\tensor_{D_{12}} \chi_s$	
	\end{enumerate}
where $\chi_s$ is the 1-dimensional representation fixed by rotation generator $r$ and alternating on reflection generator $s$, $V_{120}$ is the 2-dimensional irreducible in which $r$ corresponds to $120^\circ$ rotation, and $\mathsf{W}_{\pm}$ is a $\mathbb{Z}_2$ graded sum of Whitehouse modules $\mathsf{W}_{n,\ast}$ concentrated in degrees $n$ and $n-1$.
\end{theorem}



We have collected these statements together to highlight their uniformity, but we hasten to add that $\mathsf{gr}_0(\mathsf{FT}_\Gamma)$ can be identified with commutative graph homology, after which statement (3) was proven in \cite{BCGY} using a novel chain model for the one point compactification of $F(S^1\vee S^1,n)$.  We will prove statement (1) directly as an application of Lemma $\ref{decomplem}$.  The most novel result here, and our primary motivation, is statement (2) which has no precursors that we're aware of.  This statement  is proven as a corollary of a stronger statement (Theorem $\ref{mainthm}$ below) as we now describe.



\subsection{Application to $\mathsf{gr}_{11}(H_c^\ast(\op{M}_{2,n}))$.}

Our interest in the graph complexes $\mathsf{gr}_1(\mathsf{FT}_\Gamma)$ stems in part from their relation to graph complexes developed by Payne and Willwacher \cite{PW} and subsequently by Canning, Larson, Payne and Willwacher \cite{CLPW}, \cite{CLPW2} to model particular weight graded pieces of the compactly supported cohomology of the moduli space $\op{M}_{g,n}$.  In particular, taking the Feynman transform of the cohomology of the modular operad of Deligne-Mumford compactifications $\overline{\op{M}}_{g,n}$, and then restricting to a fixed internal degree gives a graph complex computing $\mathsf{gr}_{r}H^\ast_c(\op{M}_{g,n})$.  When $r=11$, this graph complex necessarily comes with a distinguished vertex of weight 11, and, following \cite{PW}, we use the notation $B(g,n,r)$ for the resulting graph complex. 


In this article we give a decomposition of the graph complex $B(2,n,r)$ in terms of the cobar construction of  $\mathbf{FA}$-modules:

\begin{theorem}\label{mainthm}  For all $n$ and $r\geq 3$ we construct an explicit quasi-isomorphism of $S_n$-modules
$$\Sigma^{r-2}B(2,n,r)  \stackrel{\sim}\hookleftarrow
\Omega(\mathsf{C}_{(2,1^{r-3})})(n) \oplus \ds\bigoplus_{i=1}^{\lfloor(n-r+1)/2\rfloor}\Sigma^{2i}\Omega(\mathsf{C}_{(1^{r-2})}\circ\mathsf{C}_{(1^{2i+1})} )(n). $$
\end{theorem}
One consequence of this theorem, when combined with statement (3) of Theorem $\ref{ftthm}$ above, is that for both $i=0$ and $i=11$, $\mathsf{gr}_{i}(H_c^\ast(\op{M}(2,n))$ can be decomposed into direct sums built (with distinct multiplicities) from the same irreducible pieces, namely the $\Phi[n,(a,b)^\ast]$, where $\ast$ denotes the conjugate partition of $(a,b)$.

Using the description of $\Phi[n,(1^m)]$ in terms of the Whitehouse modules, establishes the following weight 11 analog of  \cite[Theorem 1.1]{GH}:

\begin{corollary}\label{nonzero0}
	For every $n\geq 13$ the graded $S_n$-representation $\mathsf{gr}_{11} H_c^{n+\ast}(\op{M}_{2,n})$ contains the (non-graded) 
	$S_n$-representation $$\bigoplus_{\substack{0\leq j \leq n-12 \\ j \ \equiv \ n \text{ mod 2} }} H_{2j}(F(\mathbb{R}^3,n-1))^{\oplus 2}\tensor sgn_n\hookrightarrow \mathsf{gr}_{11} H_c^{n+3}(\op{M}_{2,n})  $$ 
	and
	$$\bigoplus_{\substack{0\leq j \leq n-13 \\ j \ \not\equiv \ n \text{ mod 2} }} H_{2j}(F(\mathbb{R}^3,n-1))^{\oplus 2}\tensor sgn_n\hookrightarrow \mathsf{gr}_{11} H_c^{n+2}(\op{M}_{2,n}).  $$ 
\end{corollary}
This result, combined with a bit of analysis of the Theorem in borderline cases tells us precisely when $\mathsf{gr}_{11} H_c^j(\op{M}_{2,n})$ is or is not zero;

\begin{corollary}\label{nonzerocor0}
	$\mathsf{gr}_{11} H_c^j(\op{M}_{2,n}) \neq 0$ if and only if $n\geq 13$ and $j \in \{n+2,n+3\}$ or $10\leq n \leq 12$ and $j=n+3$.  In particular, $\mathsf{gr}_{11} H_c^{n+1}(\op{M}_{2,n})=0$ for all $n$.
\end{corollary}
To the best of our knowledge, the vanishing of $\mathsf{gr}_{11} H_c^{n+1}(\op{M}_{2,n})=0$ is a new result.

\subsection{Connections and Future Directions}


We originally derived the results of this paper in the language of modules over (cyclic) operads, but were inspired to use the $\mathbf{FA}$-module language after encountering \cite{CLPW2},  
which describes the $\mathbf{FA}$-structure arising on the holomorphic forms $H^{k,0}(\overline{\op{M}}_{g,n})$ for low $k$.  
For the reader familiar with \cite{CLPW2} we point out that our Theorem $\ref{mainthm}$ could be rephrased in their notation to relate the graph complex which they denote by $G_{\tilde{1}^m}(1,n)$ to a direct sum of the graph complexes which they denote by $G_\lambda(0,n)$.  Although we've stated applications to weight 11,  our Theorem $\ref{mainthm}$ is valid for all $r$, so can likewise be applied to study the cohomology of the Feynman transform of the holomorphic $p$ forms, $\mathsf{gr}_{p,0}H^\ast_c(\op{M}_{2,n})$, in the cases $p=15,17,19$ after \cite{CLPW,CLPW2}.


Within this context, a novel feature of Theorem $\ref{mainthm}$ is that it achieves a reduction in the genus of a graph complex at the expense of trading the simple $\mathbf{FA}$-module $\widetilde{\mathsf{C}}_{(1^m)}$ for more general $\mathbf{FA}$-modules $\mathsf{C}_{\lambda}$. 
Such relationships are expected when considering 
the analog of Massey products in the modular operad setting  \cite{WardMP}, where higher operations are indexed by contractions of subgraphs which in general will change the genus of the graph.  


Our results, particularly Theorem $\ref{ftthm}$, suggest such Massey products can be chosen to be compatible with the $\mathbf{FA}$-module structure.  If this were known to be true, 
it would be possible to conclude that any two of the statements in Theorem $\ref{ftthm}$ imply the third.  And while this intuition certainly helped us generate the statements, we found it easier to prove the two new statements directly.  That said, the relationship between the modular operadic Massey products and the $\mathbf{FA}$-module structure is a very interesting direction for future study.  While we have no reason to believe that $\mathbf{FA}$-modules associated to Lie graph homology can be decomposed over the $\mathsf{C}_\lambda$ in genus $g\geq 3$, it is formally true that the Feynman transform of Lie graph homology with its Massey products is acyclic, and so similar reductions in genus should be possible.



 One area for future study is to consider the interaction between the Lie module structure on $\Omega(\mathsf{C}_\lambda)$ and the injections arising from the input $\mathbf{FA}$-module.  Note that this question is quite subtle, as the images of the injections do not commute with the cobar differential.  

While this naive $\mathbf{FI}$-module 
is not dg, there is another $\mathbf{FI}$-module we can consider by letting the partition and the arity grow in tandem.  This was illustrated in \cite{FW} for the family of partitions $\lambda=(1\cdc 1)$, which in turn established representation stability for the graph complexes $B(g+1,n,r)$ (equivalently $G_{\tilde{1}^r}(g,n)$).   We expect a generalization of these representation stability results when the family of simple $\mathbf{FA}$-modules $\widetilde{\mathsf{C}}_{1^r}$ is replaced with suitable stable families of $\mathsf{C}_\lambda$, modelled combinatially by adding marked legs to the distinguished vertex of $\Omega(\mathsf{C}_\lambda)$. 
 We save an investigation of these representation stability phenomena and their computational implications for future work.

\tableofcontents
{\bf Notation:} Throughout we work in the category of differential graded vector spaces over $\mathbb{Q}$.  For $n\geq 1$, write $S_n$ for the symmetric group of permutations of the set $\bar{n}:=\{1\cdc n\}$.  We consider $\bar{0}=\emptyset$.  For $k\in\mathbb{Z}$ we write $\Sigma^k$ for the degree shift operator on graded vector spaces, with convention $(\Sigma^k V)_{k+i} = V_i$.  If $\lambda$ is a partition of $m$ we write $|\lambda|=m$.  We write $\lambda^\ast$ for the conjugate partition.

\section{$\mathbf{FA}$-modules, the cobar construction and configuration spaces.}

In this section we fix some necessary background material on $\mathbf{FA}$-modules (following \cite{Powell, CLPW2}) and on the compactly supported cohomology of configuration spaces (following \cite{Pet20}, \cite{GH}).  We then connect these concepts via the cobar construction, establishing Theorem $\ref{thm1}$.


We use the following notation for four categories whose objects are finite sets:
\begin{itemize}
	\item $\mathbf{FB}$ is the category of finite sets and bijections,
	\item $\mathbf{FI}$ is the category of finite sets and injections,
	\item $\mathbf{FS}$ is the category of finite sets and surjections,
	\item $\mathbf{FA}$ is the category of finite sets and all set maps.
\end{itemize}
 Let $\mathbf{FX}$ denote any of the four categories above (so  $\mathbf{X}\in \{\mathbf{B},\mathbf{I},\mathbf{S},\mathbf{A} \}$).  We define an $\mathbf{FX}$-module to be a (covariant) functor from $\mathbf{FX}$.  In this paper $\mathbf{FX}$-modules will take values in the category of vector spaces, dg vector spaces, $\mathbb{Q}[G]$-modules (for a group $G$) or the category of groups.  We similarly define an $\mathbf{FX}^{op}$-module to be a contravariant functor from $\mathbf{FX}$ (or equivalently a covariant functor from $\mathbf{FX}^{op}$).  
 
Each category $\mathbf{FX}$ has a skeleton given by the full  subcategory consisting of the objects $\{ \bar{n} \ | \ n\in\mathbb{N} \}$.  Let us denote this skeleton by $\mathbf{NX}\subset \mathbf{FX}$.  We recall that the canonical restriction $\mathbf{FX}\text{-mod} \to\mathbf{NX}\text{-mod}$ is an equivalence of categories with inverse given by Kan extension.  We will typically describe $\mathbf{FX}$-modules simply in terms of their restriction to $\mathbf{NX}$ without further ado.  For example, an $\mathbf{FB}$-module $A$ (in Vect) is nothing more than a sequence of $S_n$-modules $A(n)$, also referred to as a symmetric sequence.


\subsection{Necessary Examples of $\mathbf{FA}$-modules.}
In this subsection we give several examples of  $\mathbf{FA}$-modules which will be used to state our main results.  We also give several elementary results relating these examples.

\begin{example}\label{comalg}  Let $A$ be a commutative, associative algebra with unit $e$.  Define a graded $\mathbf{FA}$-module $A^{\tensor}$ as follows.  On objects, declare $A^{\tensor}(\bar{n}) := A^{\tensor n}$.  On morphisms, given $f\colon \bar{n}_1\to \bar{n}_2$, we define $A^{\tensor n_1}\to A^{\tensor n_2}$ on pure tensors by declaring the entry in a position $j$ of the target to be the product of the elements appearing in entries $f^{-1}(j)$ in the source. The empty product is interpreted as $e$.  
\end{example}

\begin{example}\label{gralg}
Now suppose $A$ is graded commutative.  The definition in Example $\ref{comalg}$ extends to this case, provided that $A^{\tensor}(f)$ includes the Koszul sign arising from applying any permutation $\sigma$ satisfying $\sigma(i_1) < \sigma(i_2)$ if $f(i_1)< f(i_2)$, which in turn specifies the order of multiplication of each fiber.  Graded commutativity of the multiplication ensures the result is independent of the choice of such a $\sigma$.   In the graded case, for each integer $d$, there is an $\mathbf{FA}$-module given by restricting each tensor power $A^{\tensor n}$ to its degree $d$ summand.  We denote this $\mathbf{FA}$-module by $(A^{\tensor})_d$.  

We further remark that if $A$ is non-unital as a graded commutative algebra , the formula for $A(f)$ is still well-defined for surjections, and we may view $A^{\tensor}$, as well as each $(A^{\tensor})_d$, as an $\mathbf{FS}$-module.

\end{example}

\begin{example}\label{mex} For $m\geq 1$, define a symmetric sequence  $\mathsf{Inj}_m$ by letting $\mathsf{Inj}_m(n)$ be the span of the set of injections from $\{1\cdc m \} \hookrightarrow \{1\cdc n\}$.  We give $\mathsf{Inj}_m(n)$ the structure of an $\mathbf{FA}$-module as follows.  Given a set map $f\colon \bar{n}_1\to\bar{n}_2$ we define $$\mathsf{Inj}_m(n_1)\stackrel{f_\ast}\longrightarrow\mathsf{Inj}_m(n_2)$$ on a basis vector (injective map) $\phi$ by declaring $f_\ast(\phi)$ to be $f\circ\phi$ if this composition is an injection, and $0$ if it is not.  The verification that this assignment is compatible with composition of morphisms is straight-forward.  In the case $m=0$ our convention is to define $\mathsf{Inj}_0$ to be the constant functor $\mathbb{Q}$, with trivial $S_n$ action.	
\end{example}

\begin{example}
The $\mathbf{FA}$-module $\mathsf{Inj}_m$ may be viewed as taking values in the category of $S_m$-modules, via the permutation action on the source of the injections.  Formally, we view $\mathsf{Inj}_m$ as a  right $\mathbb{Q}[S_m]$-module via composition of functions; $\phi\cdot \sigma :=  \phi \circ \sigma$. As such, if $V$ is any (left) $S_m$-module, we may form a new $\mathbf{FA}$-module, denoted $\mathsf{Inj}_m\tensor_{S_m}V$, by composition of functors $(-\tensor_{S_m} V)\circ\mathsf{Inj}_m$.  

\end{example}

Let us record a relationship between the above mentioned examples of $\mathbf{FA}$-modules for future use.  Recall $R_g$ denotes a wedge of $g$ circles.  We regard $H^\ast(R_g)^{\tensor}$ as an $\mathbf{FA}$-module after Example $\ref{gralg}$.  We also write $\tilde{H}^\ast(R_g)$ for the reduced cohomology.

\begin{lemma}\label{eqfa}  There is an isomorphism of graded $\mathbf{FA}$-modules 
\begin{equation}\label{sw}
	H^\ast(R_g)^{\tensor } \cong \bigoplus_{m}\textsf{Inj}_m\tensor_{S_m} \tilde{H}^{\ast}(R_g)^{\tensor m}.
\end{equation}
\end{lemma}

\begin{proof}  For $m\leq n$ consider the map
	$$\textsf{Inj}_m(n)\tensor\tilde{H}^{\ast}(R_g)^{\tensor m}\to(H^\ast(R_g)^{\tensor n})_m $$
	given by
	$$\phi\tensor \vec{w} \mapsto  A^{\tensor}(\phi)(\vec{w}).$$ 
This map descends to the $S_m$-coinvariants which, upon checking the dimensions, is seen to be an isomorphism.  Verification that this isomorphism preserves the $\mathbf{FA}$-modules structures follows directly from the definitions in Examples $\ref{gralg}$ and $\ref{mex}$.	
\end{proof}

We find it convenient to fix some notation for the previous class of examples so, given $m$, we define $\mathsf{C} = \mathsf{Inj}_m \tensor_{S_m} -$, viewed as a functor from $S_m$-modules to $\mathbf{FA}$-modules. 
For example, if we identify $\tilde{H}^\ast(R_n)=H^1(R_n)$ with the permutation representation $\Sigma P_n$, shifted to formally live in degree $1$, then the previous lemma can be restated in this notation to say
\begin{equation}\label{permex}
 (H^\ast(R_g)^{\tensor })_m = \mathsf{C}((\Sigma P_g)^{\tensor m}).
 \end{equation} 
When $V=V_\lambda$ is an irreducible $S_m$-representation (corresponding to a partition $\lambda$ of $m$) we furthermore define
$$\mathsf{C}_\lambda:=\mathsf{C}(V_\lambda)$$
for the associated $\mathbf{FA}$-module.

\begin{example} We define an $\mathbf{FA}$-module $\tilde{\mathsf{C}}_{1^m}$ as follows.  First note that $\mathsf{C}_{1^m}(m+1) = V_{2,1^{m-1}}\oplus V_{1^{m+1}} $.   This decomposition induces a non-zero map of FA-modules $\mathsf{C}_{1^{m+1}}\to \mathsf{C}_{1^{m}}$ and we define $\tilde{\mathsf{C}}_{1^{m}}$ to be its cokernel.  The $\mathbf{FA}$-module 	$\tilde{\mathsf{C}}_{1^m}$	is simple.  We recall $\mathsf{C}_\lambda$ is simple when $\lambda \neq 1^m$ (\cite{JWG}, \cite{Powell}, \cite{CLPW2}) and moreover that every simple $\mathbf{FA}$-module is either of the form $\tilde{\mathsf{C}}_{1^{m}}$ or $\mathsf{C}_\lambda$ for such a $\lambda$. 
\end{example}

If $V_\mu$ is an $S_n$ representation and $V_\nu$ is an $S_m$ representation, let $V_{\mu}\circ V_\nu$ denote the induced $S_{n+m}$ representation.  It will be convenient to introduce the following notation:
\begin{equation}\label{circnotation}
	\mathsf{C}_{\mu}\circ\mathsf{C}_{\nu}:= \mathsf{C}(V_{\mu} \circ V_\nu).
\end{equation}
We recall that determining the irreducible decomposition of $V_{\mu}\circ V_\nu$, and hence that of $\mathsf{C}_{\mu}\circ\mathsf{C}_{\nu}$ is achieved by the Littlewood-Richardson rule, see e.g.\ \cite[Appendix A]{FH}.  We will make use of the following special case -- if $r\geq s$ then 
$
\mathsf{C}_{(r)}\circ\mathsf{C}_{(s)} = \bigoplus_{j=0}^s \mathsf{C}_{(r+j,s-j)}.
$

\begin{example}  Let $G$ be a group and let $K$ be a $G\times S_m$-module.  View the $\mathbf{FA}$-module $\mathsf{C}(K)$ as taking values in the category of $\mathbb{Q}[G]$-modules.  For each $i$, the group cohomology $H^i(G,-)$ is a functor on the category of $\mathbb{Q}[G]$-modules and so, by composition of functors, we may view $H^i(G,\mathsf{C}(K))$ as an $\mathbf{FA}$-module in vector spaces.
\end{example}

\begin{lemma}\label{pulloutc} There is a canonical isomorphism of $\mathbf{FA}$-modules
		$H^i(G,\mathsf{C}(K)) = \mathsf{C}(H^i(G,K)).$
\end{lemma}	

\begin{proof} Let $P_\ast$ be a projective resolution of $\mathbb{Q}$ in the category of $\mathbb{Q}[G]$-modules.  Write $\mathsf{Inj}^\ast_m$ for the $\mathbf{FA}^{op}$-module formed by composition of $\mathsf{Inj}_m$ with linear dualization.  Then
$$Hom_{G}(P_\ast, \mathsf{C}(K)) = Hom_{G}(P_\ast, \mathsf{Inj}_m\tensor_{S_m} K) =  Hom_{G}(P_\ast, Hom_{S_m} (\mathsf{Inj}^\ast_m, K)). $$
On the other hand
$$
\mathsf{C}(Hom_{G}(P_\ast,K))=\mathsf{Inj}_m\tensor_{S_m}(Hom_{G}(P_\ast, K)) 
= Hom_{S_m}(\mathsf{Inj}^\ast_m, Hom_{G}(P_\ast,  K) ).$$

And both functors are naturally isomorphic to $Hom_{G\times S_m}( P_\ast\boxtimes \mathsf{Inj}^\ast_m,K).$  Passing to cohomology proves the result.
\end{proof}

\subsection{The Cobar Construction for $\mathbf{FS}$ and $\mathbf{FA}$ modules.}\label{cobarsec}

In this subsection we define the cobar construction for $\mathbf{FS}$-modules. For the reader familiar with operads, we recall that $\mathbf{FS}$-modules are equivalent to comodules over the (non-unital) commutative cooperad \cite{KM,Fresse}, 
 and the (co)bar construction which we define here is simply the translation of the (co)bar construction for (co)modules over (co)operads into the language of $\mathbf{FS}$-modules.

\subsubsection{Conventions for Trees and Graphs}\label{graphs} 
We take the definition of a graph $(V,F,a,\iota)$ as having a non-empty finite set of vertices $V$ and a set of half-edges (or flags) $F$.  Each half-edge is adjacent to a unique vertex, specified by the function $a\colon F\to V$.  The number $|a^{-1}(v)|$ is called the valence of the vertex.   The final datum is an involution $\iota\colon F\to F$; given $h$, if $h\neq \iota(h)$, these two half-edges form an edge, else $h$ is called a leg.  The set of edges is denoted $E$.

Given a graph, we generate an equivalence relation on $V$ by equating $a(\iota(h))\sim a(h)$ over all $h\in F$.  We say the graph is connected if this equivalence relation has one connected component.  A connected graph such that $|V|-|E|=1$ is called a tree.    By an $n$-tree, we refer to a tree with $n$ legs, along with a labeling (aka a bijection) of the legs by the set $\{1\cdc n\}$.

A pointed $n$-tree $(\mathsf{t},v)$ is the data of an $n$-tree $\mathsf{t}$ along with a choice of vertex $v\in V(\mathsf{t})$. We call $v$ the distinguished vertex.  We remark that a pointed $n$-tree is equivalent to what is often called a rooted $n$-tree, thinking of the distinguished vertex as the root vertex, although we don't assume the existence of a root leg. A pointed $n$-tree $(\mathsf{t},v)$ is stable if $|a^{-1}(w)| \geq 3$ for all $w\in V\setminus v$.  In particular, we assert no stability condition at the distinguished vertex.  From here on all pointed $n$-trees are stable unless otherwise stated.  


An isomorphism of graphs is a pair of bijections between the respective vertices and flags which commute with the adjacencies and involutions.  Morphisms of pointed $n$-trees are presumed to preserve both the leg labels and distinguished vertex.  We write $n\text{-Tree}_\ast$ for the set of isomorphism classes of stable pointed $n$-trees.  We will typically denote an isomorphism class simply as $(\mathsf{t},v)$ without further ado.

Finally, we consider edge contraction.  Let $(\mathsf{t},v) \in n\text{-Tree}_\ast$.  For any edge $e\in E(\mathsf{t})$, we form a new stable pointed $n$-tree by contracting $e$.  Specifically this new tree has two fewer flags, and one less vertex than $\mathsf{t}$, the two vertices adjacent to $e$ having been identified.  The property of begin the distinguished vertex is absorbing upon edge contraction.  We call this new pointed $n$-tree $(\mathsf{t}/e, \bar{v})$.

\subsubsection{(Co)bar construction for $\mathbf{FS}$-modules.}

\begin{definition}  Let $A$ be an $\mathbf{FS}^{op}$-module and let $(\mathsf{t},v)\in n\text{-Tree}_\ast$. 
	We define 
	$$
	A(\mathsf{t},v) =  A(a^{-1}(v)).
	$$
\end{definition}
For such an $A$ and $(\mathsf{t},v)$, the $\mathbf{FS}^{op}$-module structure gives us a map $$f_e\colon A(\mathsf{t},v)\to A(\mathsf{t}/e,\bar{v})$$ for each $e\in E(\mathsf{t})$.  Specifically, this map is the identity if $e$ is not adjacent to $v$.  If $e$ is adjacent to $v$ then the edge contraction decomposes $a^{-1}(\bar{v})$ into the disjoint union of those flags which were originally adjacent to $v$ and those which are newly adjacent to $\bar{v}$. 
 This in turn 
 specifies a surjection $a^{-1}(\bar{v}) \to a^{-1}(v)$ fixing the former subset and sending the newly adjacent flags to the flag adjacent to $v$ which belongs to the contracted edge $e$. The map $f_e$ is the image via $A$ of this surjection.

Given an $n$-tree $\mathsf{t}$ we define $det(\mathsf{t})$ to be the top exterior power of the span of the set of edges of $\mathsf{t}$.  This is a one-dimensional vector space concentrated in degree $|E|$.  For each $e\in E(\mathsf{t})$ fix an isomorphism $g_e\colon det(\mathsf{t}) \to det(\mathsf{t}/e)$ by the formula $g_e(x\wedge e) = x$ for all $x \in det(\mathsf{t}/e)$ (identifying $E(\mathsf{t}/e)\cong E(\mathsf{t})\setminus e$).  

\begin{definition}  Let $A$ be an $\mathbf{FS}^{op}$-module and let $(\mathsf{t},v)\in n\text{-Tree}_\ast$.  Associated to this data is a degree $-1$ linear map,
$$
d_e\colon A(\mathsf{t},v)\tensor det(\mathsf{t}) \to A(\mathsf{t}/e,\bar{v})\tensor det(\mathsf{t}/e)
$$
defined by $d_e:= f_e\tensor g_e$, which we call the edge contraction map.
\end{definition}

\begin{definition}  Let $A$ be an $\mathbf{FS}^{op}$-module. The bar construction of $A$, denoted $\mathsf{B}(A)$, is a sequence of chain complexes $\mathsf{B}(A)(n)$.  The underlying graded vector spaces are defined by:
	$$
	\mathsf{B}(A)(n) = \ds\bigoplus_{\substack{ (\mathsf{t},v)} \in n\text{-Tree}_\ast} A(\mathsf{t},v)\tensor det(\mathsf{t}).
	$$
The differential is the linear extension of the map $d_{(\mathsf{t},v)}:= \sum_{e\in E(\mathsf{t})} d_e$.
\end{definition}

We will more often be interested in the cobar construction. Under the assumption that each $A(X)$ is finite dimensional, we define it as follows:

\begin{definition}  Let $A$ be an $\mathbf{FS}$-module.  The cobar construction of $A$, denoted $\Omega(A)$ is defined to be the linear dual of the bar construction of the $\mathbf{FS}^{op}$-module $A^\ast$.
\end{definition}
In particular, 
	$$
\Omega(A)(n) = \ds\bigoplus_{\substack{ (\mathsf{t},v)} \in n\text{-Tree}_\ast} A(\mathsf{t},v)\tensor det(\mathsf{t}),
$$
with differential given by expansion of edges using the $\mathbf{FS}$-module structure.

We remark that $\Omega(A)(n)$ is naturally a symmetric sequence (aka $\mathbf{FB}$-module); the $S_n$ action permutes the leaves of $(\mathsf{t}_1,v_1)$ which, in general, returns a different isomorphism class $(\mathsf{t}_2,v_2)$ as well as a bijection  $a^{-1}(v_1) \to a^{-1}(v_2)$.  We also remark that the operadic perspective implies that the symmetric sequence $\Omega(A)$ assembles to a module over the shifted $L_\infty$ operad.

\begin{definition}  If $A$ is an $\mathbf{FA}$-module, we define the cobar construction of $A$, still denoted $\Omega(A)$, to be the cobar construction of the underlying $\mathbf{FS}$-module.
\end{definition}

In particular, the cobar construction that we consider, while often taking an $\mathbf{FA}$-module as an input, does not use the injections when forming the differential.

\subsubsection{Combinatorial Examples}

One nice feature of the (co)chain complexes $\Omega(\mathsf{Inj}_m)$ and $\mathsf{B}(\mathsf{Inj}^\ast_m)$ is that they may be described completely combinatorially.  For example, each chain complex $\mathsf{B}(\mathsf{Inj}^\ast_m)(n)$ is spanned by pointed $n$-trees $(\mathsf{t},v)$, along with the choice of injection $\phi\colon\bar{m}\to a^{-1}(v)$; we depict this data by writing the number $i$ labeling a tic mark (marking) on the flag $\phi(i)$.  The differential is the sum over edge contractions; if an edge contracts a flag with a labeled marking we sum over ways to distribute this marking to the newly adjacent flags.  
See Figure $\ref{fig:f3}$.

\begin{figure}
	\centering
	\includegraphics[width=1\linewidth]{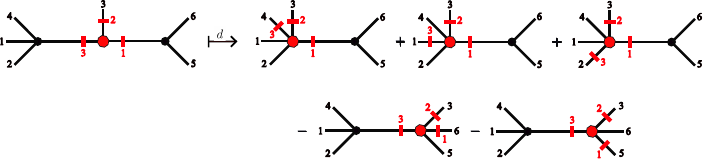}
	\caption{The differential in $\mathsf{B}(\mathsf{Inj}^\ast_3)(6)$ contracts edges and distributes ordered markings.  The signs presuppose that we've ordered the edge containing the tic mark \textcolor{red}{1} before the edge containing the tic mark \textcolor{red}{3}.}
	\label{fig:f3}
\end{figure}

For the $\mathbf{FA}$-modules $\mathsf{C}_\lambda$, there isn't as simple a combinatorial/graphical presentation of the (co)bar construction in general.  But let us take note of several examples which do have a particularly simple description.  First, when $\lambda = (1^m)$ both $\widetilde{\mathsf{C}}_\lambda$ and $\mathsf{C}_\lambda$ have (co)bar constructions which can be described as trees with unlabeled tic marks.  For the former, the combinatorics of the bar construction were described directly in \cite{WardStir};  a weakly equivalent combinatorial complex formed via the Koszul resolution is considered in \cite{PW}.  

There is also a very natural intermediary given by trees with unlabeled markings separated into a fixed number of subsets, which we indicate by markings of different colors.  We will have particular use for the case of two colors, so we make the following definition:


\begin{definition}\label{2colordef} For natural numbers $n,b,r$ we define the chain complex: 
	$$\mathsf{T}(n,b,r) := 
	\mathsf{B}(\mathsf{C}^\ast_{(1^b)}\circ\mathsf{C}^\ast_{(1^r)})(n)$$
with notation $\circ$ as in Equation $\ref{circnotation}$.
\end{definition}
In particular, the chain complex $\mathsf{T}(n,b,r)$ can be described as trees with two colors (blue and red say) of unlabeled markings.  The markings carry degree $-1$ and the alternating $S_b$ (resp. $S_r$) action.  Contraction of an edge containing a marking (of either color) sums over ways to distribute the marking, as above. See the bottom line of Figure $\ref{fig:fig1}$ for an example.

\subsection{Configuration spaces via the cobar construction.}

Finally in this section we reinterpret the results of \cite{Pet20} and \cite{GH} in terms of the cobar construction of $\mathbf{FS}$-modules and $\mathbf{FA}$-modules.

\subsubsection{CE complex}

We write $\mathsf{Lie}$ for the Lie operad and write $\mathsf{sLie}$ for a shift of its operadic suspension, in particular $\mathsf{sLie}(n) := \Sigma^{n}\mathsf{Lie}(n)\tensor sgn_n$ is a graded vector space of dimension $(n-1)$! concentrated in degree $n$, spanned by the set of Lie words on $n$ distinct letters, modulo Jacobi and anti-symmetry relations.

Let $A$ be a dg commutative algebra.  Then $A\tensor \mathsf{sLie}:=\oplus_n (A\tensor \mathsf{sLie}(n))$ is naturally a left module over the Lie operad via the formula $[a_1\tensor \ell_1, a_2\tensor \ell_2] = a_1a_2\tensor (-1)^{|a_2||\ell_1|}[\ell_1,\ell_2]$, where $[\ell_1,\ell_2]$ denotes the operation induced by bracketing two Lie words together, with suitable reindexing.

We may take the Chevalley-Eilenberg chain complex associated to this Lie module $\mathsf{CE}_\ast(A\tensor \mathsf{sLie})$ and observe that the differential $d_{\mathsf{CE}}$ does not alter the total number of letters, $n$, given in an exterior product of Lie words.  Thus, this complex splits over $n$, each component having an $S_n$ action, and we may view $\mathsf{CE}_\ast(A\tensor \mathsf{sLie})$ as a symmetric sequence.  For the grading, we give the exterior factors degree $-1$; in particular if we write
$\mathsf{CE}_p(A\tensor \mathsf{sLie})(n)$ for those chains having $p$ exterior factors and $n$ total letters, we view this vector space as concentrated in degree $n-p$.  So we formally get a cochain complex, opposite to the more standard grading on the $\mathsf{CE}$ chain complex.

\begin{lemma}\label{CEthm}  Let $A$ be a dg commutative algebra, not necessarily unital, and consider $A^{\tensor}$ as an $\mathbf{FS}$-module as in Example $\ref{gralg}$.   Then there is a natural quasi-isomorphism of dg symmetric sequences
	$$\eta\colon\Omega(A^{\tensor})\stackrel{\sim}\to  \mathsf{CE}_\ast(A\tensor \mathsf{sLie}).$$
\end{lemma}
\begin{proof} 

Fix $n$.  Recall that 
 $\Omega(A^{\tensor})(n)$ is spanned by pointed $n$-trees $(\mathsf{t},v)$ with elements of $A$ labeling the set $a^{-1}(v)$.  Let us call the components in the forest formed by removing $v$ the ``branches'' of the tree. A single branch is either a stable rooted tree, or a single lone leg. The set of branches partitions $\bar{n}$ via the associated leg labels.  
 
 The map $\Omega(A^{\tensor})(n)\stackrel{\eta}\to \mathsf{CE}_\ast(A\tensor \mathsf{sLie})(n)$ sends a summand indexed by a given tree to zero unless all its neutral vertices are trivalent, in which case it reads each branch as a flow chart, with a Lie bracket at each neutral vertex, producing a list of Lie words whose letters partition $\bar{n}$.  Any order of this list of branches produces a corresponding order on the set of Lie words, hence an element in $\mathsf{CE}_\ast(A\tensor \mathsf{sLie})(n)$.   Since both the edges of the tree and the exterior factors of $A\tensor \mathsf{sLie}$ are alternating, this assignment is independent of the chosen order.  
 
 We next check that this map is dg.  The internal differential $d_A$ commutes with $\eta$, so we check compatibility with the remaining differential terms.  Edge expansions at the distinguished vertex are mapped by $\eta$ to $0$ unless they expand and merge exactly two branches, which in turn corresponds to the $\mathsf{CE}$-differential.  Expansions at neutral vertices, if possible, are sent to zero unless we consider a tree having one non-trivalent, neutral vertex of valence $4$.  In this case, expanding produces a sum of three possible configurations of the branch which map to $0$ in the target thanks to the Jacobi identity.

So it remains to verify that $\eta$ is a quasi-iso.  For this, filter the mapping cone of $\eta$ $\Sigma^{-1}\Omega(A^\tensor)(n)\oplus\mathsf{CE}_\ast(A\tensor \mathsf{sLie})(n)$ by saying a sum $(a,b)$ is in filtration degree $q$ if the tree supporting $a$ has at most $n-q$ branches and the CE chain $b$ has degree at least $q$ (hence consists of at most $n-q$ exterior factors).  To see this is a filtration, note the number of branches and exterior factors can't increase upon applying the differential (hence filtration degree can't decrease). 

The differential on the associated graded must preserve the number of branches on the one hand, and the number of exterior factors on the other.  Therefore, the associated graded splits over partitions of $\bar{n} = \cup_i b_i$ into $n-q$ blocks.  Each summand in this splitting is of the form $A^{\tensor n-q }$ tensored with $\tensor_{i} \mathsf{sL}_\infty(b_i) \stackrel{\eta}\twoheadrightarrow \tensor_{i}\mathsf{sLie}(b_i),$ hence is acyclic by Koszulity of the commutative and Lie operads.
\end{proof}

Applying Petersen's Theorem \cite[Corollary 8.8]{Pet20} to the above yields:
\begin{corollary}\label{DPcor}  Let $X$ be a paracompact and locally compact Hausdorff space.  Let $A$ be a cdga model for the compactly supported cochains of X over $\mathbb{Q}$.  Then there is an isomorphism of graded symmetric sequences
	$$H^\ast(\Omega(A^{\tensor})) \cong  H_c^\ast(F(X,-)).$$
\end{corollary}

\subsubsection{Wedges of circles.}
Next, we specialize to the case $X=R_g$, a wedge of $g$ circles.  We let $A= H^\ast(R_g)$, which we note is a cdga model for the compactly supported cochains on $X$.  Recall (Lemma $\ref{eqfa}$) there is an equivalence of $\mathbf{FA}$-modules  
$A^{\tensor } = \bigoplus_{m}\textsf{Inj}_m\tensor_{S_m} \tilde{H}^{\ast}(R_g)^{\tensor m}.$  Applying $\Omega$ we have:

\begin{proposition}\label{takecobar}  Let $A = H^\ast(R_g)$.  Then there is an equivalence of dg $S_n$-modules
	$$\Omega(A^{\tensor})(n) \cong \bigoplus_{n\geq m}\Omega(\textsf{Inj}_m)(n)\tensor_{S_m} \tilde{H}^{\ast}(R_g)^{\tensor m}$$
\end{proposition}

In light of this proposition, we define the graded $S_n\times S_m$-module:
$$
\Phi[n,m] := H^\ast(\Omega(\mathsf{Inj}_m))(n)
$$
Then, combining Proposition $\ref{takecobar}$ with Corollary $\ref{DPcor}$ we find:

\begin{corollary}\label{checkGH}  There is an isomorphism of graded symmetric sequences:
$$
H_c^\ast(F(R_g,n))\cong \bigoplus_{n\geq m}\Phi[n,m]\tensor_{S_m} \tilde{H}^{\ast}(R_g)^{\tensor m}.$$
\end{corollary}
This result follows from \cite[Theorem 1.5]{GH}, modulo verification that our definition of $\Phi[n,m]$ coincides with theirs.  For this, we may appeal to the uniqueness of coefficients in a polynomial functor.  In particular, given a vector space $W$ we may form a unital graded commutative algebra $A_W:=\mathbb{Q}\oplus \Sigma W$ (with $w_1w_2=0$ forced by degree considerations).  Then the two descriptions of $\Phi[n,m]$ are coefficients of the functor $W\mapsto H^\ast(\Omega((A_W)^{\tensor}))(n)$ by Lemma $\ref{CEthm}$.

\begin{definition} Define the graded $S_n$-module:
$$
\Phi[n,\lambda] := H^\ast(\Omega(\mathsf{C}_\lambda))(n).
$$
\end{definition}
In particular, each  $\Phi[-,\lambda]$ is naturally a module over the Lie operad. 
In general, the cohomology $\Phi[n,\lambda]$ is unknown, but we will make use of the following known computations:
\begin{example}\label{WHex}  Write $F(\mathbb{R}^3,n-1)$ for the configuration space of $n-1$ points in $\mathbb{R}^3$.  Recall this space has cohomology only in even degrees $0\leq 2i \leq 2(n-2)$, and that each $H^{2i}(F(\mathbb{R}^3,n-1))$ may be viewed as an $S_{n}$-module via an identification with $F(S^3,n)/SU(2)$.  Then \cite[Proposition 5.16]{GH} show that $\Phi[n,(1^m)]\tensor sgn_n$ is isomorphic to $H_{2(n-m)}(F(\mathbb{R}^3,n-1))$ in degree $n-m$, is isomorphic to $H_{2(n-m-1)}(F(\mathbb{R}^3,n-1))$ in degree $n-m-1$ (interpreted as $0$ in the case $n=m$) and is $0$ elsewhere.\end{example}


Because it will recur in our results below, let us adopt some notation for the graded $S_n$-representation 
\begin{equation}\label{Wdef}
	\mathsf{W}_{\pm}(n) :=\bigoplus_{i} \Sigma^{2i}\Phi[n,(1^{2i})].
\end{equation} Note that, using the Betti numbers of the configuration space of points in $\mathbb{R}^3$, we see this gives two distinct $(n-1)!/2$-dimensional representations of $S_n$ in degree $n$ and $n-1$ respectively.  Let us denote the representation in degree $n$ as $\mathsf{W}_+(n)$ and the representation in degree $n-1$ as $\mathsf{W}_-(n)$.

We remark that \cite{FW} shows (in the present notation) that $\mathsf{W}_-(n)\cong H_\ast(\Delta_{1,n})$, the tropical moduli space of genus $1$, or equivalently to the homology of the graph complex $\mathsf{GC}_2(1,n)$.  We suspect that $\mathsf{W}_+(n)$ is isomorphic to the unique non-zero homology group of the opposite parity complex $\mathsf{GC}_3(1,n)$, although we won't use this fact.

\section{$\mathbf{FA}$-modules and Lie graph homology.}  

The aim of this section is to first show that $\mathbf{FA}$-modules arise when considering Lie graph homology.  Using work of \cite{CHKV} we arrive at an explicit decomposition over the $\mathbf{FA}$-modules of the form $\mathsf{C}_\lambda$.      
With this decomposition, we can make a direct comparison between the Feynman transform of Lie graph homology in low genus and the compactly supported cohomology of $F(S^1\vee S^1; -)$.  In particular, we prove Statement 1 of Theorem $\ref{ftthm}$ from the introduction.

\subsection{Cobar Construction of $\mathbf{FS}$-modules and the Feynman transform}  To begin, we very briefly recall the notions of modular operads and the Feynman transform, introduced in \cite{GeK2}.  We then record the relationship to modules over operads and their bar construction.  Finally, we specialize to modular operads under the commutative operad, which give examples of $\mathbf{FS}$ and $\mathbf{FA}$-modules, whose (co)bar constructions may be used to probe a portion of the Feynman transform.  Additional references for modular operads and the Feynman transform which are particularly well calibrated for these results include \cite{WardMP} and \cite{CLPW2}.

\subsubsection{Recollection of Modular Operads}
 Define a modular graph to be a connected graph $\gamma$ along with an ordering on its sets of legs and a function $g\colon V(\gamma) \to \mathbb{N}$, such that each vertex satisfies the stability condition $|a^{-1}(v)| + 2g(v) \geq 3$.  We call $g(v)$ the genus of the vertex $v$.  The total genus of a modular graph is defined to be $g(\gamma) = \beta(\gamma) +\sum_v g(v)$, where $\beta(\gamma) = |E|-|V|+1$ is the first betti number of the connected graph.  A modular graph with $n$ legs and total genus $g$ is said to be of type $(g,n)$.  
 
 A modular operad $\op{M}$ is a sequence of symmetric sequences $\op{M}(g,-)$ having contraction operations indexed by modular graphs.  If $\gamma$ is a modular graph of type $(g,n)$ the corresponding operation is an $S_n$-equivariant map
$\mu_\gamma\colon \op{M}(\gamma) := \tensor_v \op{M}(g(v),a^{-1}(v))\to \op{M}(g,n)$.
Each such operation may be generated by a sequence of single edge contractions.  These generating operations come in two types depending on whether a given edge is adjacent to one vertex or two (i.e.\ is or is not a loop):
$$
\op{M}(g,n)\stackrel{\circ_{ij}}\to \op{M}(g+1,n-2) \ \ \ \ \ \   \op{M}(g_1,n_1)\tensor \op{M}(g_2,n_2) \stackrel{\xycirc{i}{j}}\to \op{M}(g_1+g_2,n_1+n_2-2).
$$
We call $\circ_{ij}$ ``self-gluings'' and $\xycirc{i}{j}$ ``non-self gluings''.

If $\op{M}$ is a modular operad, then we may form an associated, underlying (cyclic) operad $\op{O}$ by setting $\op{O}(n):=\op{M}(0,n+1)$ and retaining the relevant non-self gluings of the modular operad $\op{M}$ to give the operad structure on $\op{O}$.  For each $g$, we may furthermore choose to view the symmetric sequence $\op{M}(g,-)$ as a right module over the operad $\op{O}$ via the relevant non-self gluings.

An important variant of the notion of modular operads are $\mathfrak{K}$-twisted modular operads, in which the generating operations are indexed as above, but each is assumed to have degree 1.

\subsubsection{Recollection of the Feynman Transform.}
The Feynman transform is a pair of functors, both denoted by $\mathsf{FT}$, between the categories of modular and $\mathfrak{K}$-twisted modular operads.  For our purposes, it will also be convenient to consider the linear dual notion of the co-Feynman transform, which we denote $\mathsf{cFT}$, taking (twisted) modular operads to cooperads and vice-versa.

If $\op{M}$ is a modular operad we define the coFeynman transform to be the family of chain complexes:
$$
\mathsf{cFT}(\op{M})(g,n) = \ds\bigoplus_{[\gamma]} \op{M}(\gamma)\tensor_{Aut(\gamma)} det(\gamma)
$$
where the sum is indexed by isomorphism classes of modular graphs of type $(g,n)$, and where the differential is given by a sum over edge contractions.  As above, $det(\gamma)$ denotes the top exterior power of the set of edges; the appearance of this factor will ensure the differential is indeed square zero.  In the case where $\op{M}$ is $\mathfrak{K}$-twisted, the formula is the same except the $det(\gamma)$ factor is removed.  

Operating under the simplifying assumption that our graded vector spaces are finite dimensional in each graded component, we define the Feynman transform $\mathsf{FT}$ to be the linear dual of the coFeynman transform.  Grafting legs of the indexing graphs ensures the image of the Feynman transform is itself a modular operad.  This original definition of \cite{GeK2} suits our purposes; if our modular operads were not finite dimensional we could choose to define the Feynman transform to take modular cooperads as inputs and avoid linear dualization.

\subsubsection{Connection to the Cobar Construction of $\mathbf{FS}$-modules}

In the co-Feynman transform, we can identify a subcomplex of each $\mathsf{cFT}(\op{M})(g,n)$ by summing only over those graphs having a vertex $v$ satisfying $g(v)=g$.  This condition forces the graph to be a tree, with only this lone vertex $v$ having non-zero genus.  In other words, the modular graph $\gamma$ is a pointed $n$-tree, with distinguished vertex $v$.  As such, the bar construction of $\op{M}(g,-)$, viewed as a right module over the operad $\op{M}(0,-)$, specifies a sequence of subcomplexes of $\mathsf{cFT}(\op{M})(g,-)$.

Let us now specialize further to the case that the operad $\op{M}(0,-)$ contains the commutative operad $\mathsf{Com}$ as a (cyclic) sub-operad (this includes the case when $\op{M}(0,-) =\mathsf{Com}$ of course).  By considering only those non-self gluings between $\mathsf{Com}$ and $\op{M}(g,n)$ we find:

\begin{lemma}\label{FSbar}
Let $\op{M}$ be a modular operad for which $\mathsf{Com}\subset \op{M}$.  For each $g$, the symmetric sequence $\op{M}(g,-)$ inherits the structure of an $\mathbf{FS}^{op}$-module; applying the cobar construction gives a sequence of subcomplexes
$$	
\mathsf{B}_{\mathbf{FS}}(\op{M}(g,-))\subset \mathsf{cFT}(\op{M})(g,-).
$$
Dually, $\op{M}(g,-)^\ast$ inherits the structure of an $\mathbf{FS}$-module; applying the cobar construction gives a sequence of quotient complexes
$$	
\mathsf{FT}(\op{M})(g,-) \twoheadrightarrow \Omega_{\mathbf{FS}}(\op{M}(g,-)^\ast).
$$
\end{lemma}

One example of such an $\op{M}$ is the modular operad formed by the homology of Deligne-Mumford compactifications of the moduli space of curves, for which the associated $\mathbf{FA}$-modules are considered in \cite{CLPW2}.  Our access point came from another example of such a modular operad, namely Lie graph homology, which we now recall. 

\subsection{$\mathbf{FA}$-module decomposition of Lie Graph Homology  for $g=1,2$.}

Recall $\mathsf{sLie}$ denotes a shift of the suspension of the (cyclic) operad $\mathsf{Lie}$.  Extending by $0$ to higher genus, $\mathsf{sLie}$ forms a $\mathfrak{K}$-twisted modular operad and we may take its Feynman transform.  We refer to the homology of this modular operad $H_\ast(\mathsf{FT}(\mathsf{sLie}))$ as Lie graph homology.  Koszul duality between the commutative and Lie opeards implies that the underlying cyclic operad in genus $0$ is the commutative operad $\mathsf{Com}= H_\ast(\mathsf{FT}(\mathsf{sLie}))(0,-)$.

As mentioned in the introduction,
$\mathsf{FT}(\mathsf{sLie})(g,n)$ computes the homology of a group $\Gamma_{g,n}$ which can be described, after \cite{CHKV}, as the homotopy classes of homotopy automorphisms of a wedge of $g$ circles which fix $n$ given points.  
 As such, from here on, we write $H_\ast(\Gamma)$ for the modular operad $$H_\ast(\Gamma):=H_\ast(\mathsf{FT}(\mathsf{sLie})).$$  We also write $H^\ast(\Gamma)$ for the dual modular cooperad.


From the above discussion we know that the for each $d$ and $g$ the symmetric sequence $H^d(\Gamma_{g,-})$ has the structure of an $\mathbf{FS}$-module coming from the modular operadic compositions which compose along trees with a distinguished genus $g$ vertex.   
The purpose of this section is to prove that these $\mathbf{FS}$-modules may be extended to $\mathbf{FA}$-modules when $g=1$ and $g=2$, and to give their explicit decomposition.

Let's first record the case of $g=1$.  Although not presented in this language, this case can be extracted from our earlier work:

\begin{lemma}\cite[Lemma 3.2]{WardW}\label{g1fa}  The $\mathbf{FS}$-module structure on
	$H^m(\Gamma_{1,\ast})$ lifts to an $\mathbf{FA}$-module given by $\widetilde{\mathsf{C}}_{1^{m+1}}$ when $m\geq 2$ is even and 0 otherwise.
\end{lemma}

Our main interest is then the case $g=2$.   Since our interest is in the multiplicity of the $\mathsf{C}_\lambda$, we give the statement in additive notation, writing  $\omega\mathsf{C}_\lambda$ in place of $\mathsf{C}_\lambda^{\oplus \omega}$. 

\begin{proposition} \label{g2dec}  Let $d>0$.  The $\mathbf{FS}$-module structure on $H^d(\Gamma_{2,\ast})$ lifts to an $\mathbf{FA}$-module given by:	
	\begin{itemize}
		\item  If $d\equiv 2 \text{ mod } 4$ then $H^d(\Gamma_{2,\ast}) = 0$.
		\item  If $d\equiv 0 \text{ mod } 4$ then $H^d(\Gamma_{2,\ast}) = 
		\mathsf{C}_{(2^{d/2})}$.
		\item If $d=q+1$ is odd then 
	$$H^{d}(\Gamma_{2,\ast}) = 
\ds\bigoplus_{0\leq i<q/2} w_{i,q}\mathsf{C}_{(2^i,1^{q-2i})}$$
where 
$$w_{i,q} = \begin{cases}
	 \text{dim}(\op{M}_{q-2i+2}) & \text{ if } i \text{ is odd,} \\ 	 
	 \text{dim}(\op{S}_{q-2i+2}) & \text{ if } i \text{ is even.} 
\end{cases}
$$
the dimensions of the space of modular (resp.\ modular cusp) forms of the given weight.
	\end{itemize}
\end{proposition}
\begin{proof}  The key ingredients are in \cite{CHKV}. 
	
For  $n\geq 1$, define the natural projection $\pi\colon \Gamma_{g,n+1}\to \Gamma_{g,n}$ as follows.  If $X_{g,n+1}$ is a graph of genus $g$ with $n+1$ labeled marked points then there is an inclusion map $hAut_\partial(X_{g,n+1}) \hookrightarrow hAut_{\partial^\prime}(X_{g,n+1}) $ where the subscript $\partial$ indicates that all the boundary points are fixed pointwise, and ${\partial^\prime}$ indicates that only the first $n$ points must be fixed.  Applying $\pi_0$ to this inclusion gives the map $\pi$.  It is shown in \cite[Proposition 2.2]{CHKV} that this map is surjective and admits a splitting, which we denote by $\epsilon\colon \Gamma_{g,n} \to \Gamma_{g,n+1}$.  Upon passage to homology, $\epsilon$ induces the assembly map which coincides with the modular operadic composition $\xycirc{0}{n}$, see \cite[Section 6.1]{CHKV}.

Now specialize to genus $g=2$.	Let $F_2$ denote the free group on two generators. 
Recall there is an isomorphism of groups $GL_2(\mathbb{Z})\cong Out(F_2)\cong \Gamma_{2,0}$, and projection $\Gamma_{2,n}\to \Gamma_{2,0}$ is part of a short exact sequence with kernel $(F_2)^n$.  Consider two such sequences and the following maps between them: 
$$
\xymatrix{	1  \ar[r] & \ar[d]^{diag} (F_2)^{n} \ar[r] &\ar[d]^{\epsilon} \Gamma_{2,n} \ar[r] & GL_2(\mathbb{Z}) \ar[r] \ar[d]^{=}& 1	\\  1 \ar[r]  & (F_2)^{n+1} \ar[r] & \Gamma_{2,n+1} \ar[r]  & GL_2(\mathbb{Z}) \ar[r]  & 1 }
$$
Here diag refers to putting the $n^{th}$ entry of the source into entries $n$ and $n+1$ in the target, with inclusion for all other entries.  One may check directly from the definition of $\epsilon$ (given in the proof of Lemma 2.2 in \cite{CHKV}) that this diagram commutes.


We then consider the Hochschild-Serre spectral sequences associated to these short exact sequences.  To simplify the notation, let us assume that $d=q+1$ is odd; the even degree case follows by the same logic. Invoking \cite[Theorem 3.10]{CHKV}, these spectral sequences degenerate at the $E^2$ page, and moreover $H^d(\Gamma_{2,n}) \cong E^2_{q,1} = H^{1}(GL_2(\mathbb{Z}), H^{q}((F_2)^n))$.  As such, naturality of the spectral sequence tells us that these isomorphisms take the $\mathbf{FS}$-module structure on  $H^d(\Gamma_{2,\ast})$ coming from the modular co-operad $H^\ast(\Gamma)$ to the $\mathbf{FS}$-module structure induced on cohomology from the $\mathbf{FS}^{op}$-module structure on the Cartesian powers of $F_2$ induced by the diagonal maps.

Upon passage to cohomology $H^q((F_2)^n)\cong (H^\ast(R_2)^{\tensor n})_q$, the diagonal map returns the cup product, and so this latter $\mathbf{FS}$-module structure coincides with that in Equation $\ref{permex}$ above (in the case $g=2$). 
After applying the isomorphism 
$$
H^1(GL_2(\mathbb{Z});\mathsf{C}( \tilde{H}^\ast(R_2)^{\tensor q})) =  \mathsf{C}(H^1(GL_2(\mathbb{Z}); \tilde{H}^\ast(R_2)^{\tensor q}))
$$
from Lemma $\ref{pulloutc}$ above, 
the Proposition reduces to the computation of 
$H^1(GL_2(\mathbb{Z}); \tilde{H}^\ast(R_2)^{\tensor q}) $, for which we appeal to \cite[Lemma 3.8]{CHKV}.
\end{proof}

\subsection{Filtration of the Feynman transform.}
Taking the cobar construction of the decomposition in Proposition $\ref{g2dec}$ allows us to study the Feynman transform of Lie graph homology.  To make this precise, we introduce the following filtration on the Feynman transform following \cite{WardMP}.

Given a modular operad $\op{M}$, we filter the Feynman transform: 
$$
F_{h}(\mathsf{FT}(\op{M})) = \ds\bigoplus_{ \sum g(v) \leq h} \op{M}(\gamma)\tensor_{Aut(\gamma)} det(\gamma).
$$
Since the differential expands edges, the statistic $\sum g(v)$ can not increase upon applying the differential, so for each $(g,n)$ we have a filtered chain complex $F_{h-1}(\mathsf{FT}(\op{M})(g,n))\subset F_{h}(\mathsf{FT}(\op{M})(g,n))$ with filtration degree bounded by $0\leq h \leq g$.

We would like to apply this filtration in the case $\op{M}= H_\ast(\Gamma)$.  Before doing so, let us recall the following technical result.  A homogeneous vector in $\mathsf{FT}(H_\ast(\Gamma))$ is given by a graph along with a class in some $H_d(\Gamma_{g(v),n(v)})$ associated to each vertex.  In this context, we say the vertex has degree $d$.  Given an edge adjacent to a single vertex, often called a tadpole, we declare the degree of tadpole to be the degree of the vertex.  Then:

\begin{lemma}\cite[Lemma 3.6]{WardLie}  The subcomplex of $\mathsf{FT}(H_\ast(\Gamma))(g,n)$ spanned by graphs containing any degree $0$ tadpoles and/or degree $0$ vertices of higher genus is acyclic.
\end{lemma}

We write $\overline{\mathsf{FT}}(H_\ast(\Gamma))(g,n)$ for the associated quotient complex, and observe that the filtration defined above descends to this quotient.  When $g=2$, let us specify the following notation for the associated graded:

\begin{definition}\label{grdef}  For $h\in\{0,1,2\}$ define a graded symmetric sequence $\mathsf{gr}_h(\mathsf{FT}_\Gamma)$ by
	$$\mathsf{gr}_h(\mathsf{FT}_\Gamma):= F_h \overline{\mathsf{FT}}(H_\ast(\Gamma))(2,-)/F_{h-1} \overline{\mathsf{FT}}(H_\ast(\Gamma))(2,-). $$
\end{definition}
Let's unpack the three cases of this definition:
\begin{itemize}
	\item $\mathsf{gr}_0(\mathsf{FT}_\Gamma)(n)$ is spanned by tadpole free graphs having $\beta(\gamma) = 2$ and $n$ legs, with each vertex labeled by $\mathbb{Q}=\mathsf{Com}(v)$.  The differential is edge expansion. This complex is quasi-isomorphic to the Feynman transform of the commutative operad in biarity $(2,n)$, and after \cite{CGP2}, computes the homology of the tropical moduli space $\Delta_{2,n}$.  
	
	\item $\mathsf{gr}_1(\mathsf{FT}_\Gamma)(n)$ is spanned by graphs having $\beta(\gamma) = 1$ and $n$ legs, along with a distinguished vertex of genus $1$, carrying a label in $H^\ast(\Gamma_{1,v})$ of non-zero degree.  The differential is expansion of non-tadpole edges.
	
	\item $\mathsf{gr}_2(\mathsf{FT}_\Gamma)(n)$ is spanned by $n$-trees ($\beta(\gamma) = 0$) of two types -- those with one distinguished vertex, carrying a label in $H^\ast(\Gamma_{2,v})$ of non-zero degree, and those with two special vertices $v_1$,$v_2$, each carrying labels in $H^\ast(\Gamma_{1,v_i})$ of non-zero degree.  The differential is again expansion of non-tadpole edges, including the possible expansion of the distinguished genus $2$ vertex into two adjacent genus $1$ vertices, preserving the total degree of the vertex label(s).
\end{itemize}

\subsection{Decomposition of the Tree Part $\mathsf{gr}_2(\mathsf{FT}_\Gamma)(n)$.}
The goal of this subsection is to derive an explicit decomposition of $H^\ast(\mathsf{gr}_2(\mathsf{FT}_\Gamma)(n))$ in terms of the $\Phi[n,\lambda]$.  To simplify the notation, in this subsection only, we write
$$\mathsf{T}_\Gamma(n) : =\mathsf{gr}_2(\mathsf{FT}_\Gamma)(n)$$
where $\mathsf{T}_\Gamma$ reminds us that we consider the span of trees, having either 1 or 2 distinguished vertices labeled by cohomology classes of the appropriate $\Gamma$.  Since the differential in the Feynman transform is by edge expansion, the trees with two genus $1$ vertices form a subcomplex.   In addition, since the Feynman transform differential preserves the total degree of the vertex labels, each chain complex $\mathsf{T}_\Gamma(n) $ splits over internal degree.  Let us denote this decomposition by  $\mathsf{T}_\Gamma(n)_d$.  

Thus, for each $d$ we have a short exact sequence of chain complexes:
\begin{equation}\label{ses}
	0 \to \mathsf{T}^{1,1}_\Gamma(n)_d\to \mathsf{T}_\Gamma(n)_d \to \mathsf{T}^{2}_\Gamma(n)_d \to 0 
\end{equation}
such that 
$\mathsf{T}^{2}_\Gamma(n)_d = \Omega(H^d(\Gamma_{2,\ast}))(n)$.

\begin{lemma}\label{uformula}  For $d$ odd there is an isomorphism of graded $S_n$-modules
	$$ 
H^\ast(\mathsf{T}_\Gamma(n)_d)\cong \bigoplus_{a+b=d-1} u_{(a,b)}\Sigma^{d}\Phi[n,(a,b)^\ast],
	$$
where
	$$u_{(a,b)} =
	\begin{cases}
		\left\lfloor \ds\frac{a-b-2}{4}\right\rfloor - \left\lfloor \ds\frac{a-b}{6}\right\rfloor+
		\delta_{a,b} & \text{ if }  a\equiv b \equiv 0 \text{ mod } 2 \\ 
		& \\
		\left\lfloor \ds\frac{a-b-2}{4}\right\rfloor - \left\lfloor \ds\frac{a-b}{6}\right\rfloor+
		1 & \text{ if }  a\equiv b \equiv 1 \text{ mod } 2 \\ 	
		& \\
		0 & \text{ else.}
	\end{cases}
	$$

\end{lemma}
\begin{proof}
	When $d$ is odd, the contribution from the genus $1$ labels vanish (Lemma $\ref{g1fa}$), and we have $\mathsf{T}_\Gamma(n)_d=\mathsf{T}^{2}_\Gamma(n)_d = \Omega(H^d(\Gamma_{2,\ast}))(n)$.  So it suffices to apply $\Omega$ to the expression in Proposition $\ref{g2dec}$.
	
	Note $(a,b)^{\ast} = (2^{b}, 1^{a-b})$.  Setting $d=q+1$ (so $q$ is even), $i=b$ and $a-b=q-2i$, we determine the coefficient $u_{a,b}$ of $\Phi[-,(a,b)^\ast]$ appearing in $H^\ast(\mathsf{T}_\Gamma(n)_d)$ to be $w_{b,a+b}$.  Since $q=a+b$ is even, this coefficient appears only if $a \equiv b$ mod $2$.  

When $a$ and $b$ are both odd, so $i$ is odd, we find $w_{b,a+b} = \text{dim}(\op{M}_{a-b+2})$.  When $a$ and $b$ are both even but distinct we find $w_{b,a+b}=\text{dim}(\op{S}_{a-b+2})$  Recalling that for $a-b-2\geq 0$, 
$$  dim(\op{M}_{a-b+2})=
1+ dim(\op{S}_{a-b+2}) = 1 + \lfloor (a-b-2)/4 \rfloor - \lfloor (a-b)/6 \rfloor
$$
we see this matches the statement.  Finally, if $a$ and $b$ are even and $a=b$, the partition $(b,b)^{\ast} = (2^b)$ should appear with coefficient $0$ according to Proposition $\ref{g2dec}$.  We add $\delta_{b,b}:=1$ in this case to account for the negative appearing in the floor function $\lfloor -2/4\rfloor +1=0$.
\end{proof}

Next we consider the case $d$ even, which we subdivide into two cases considering $d$ mod $4$.

\begin{lemma}\label{mod2}  If $d\equiv 2$ mod $4$ then there is an $S_n$-equivariant quasi-isomorphism:
	$$
 \mathsf{T}_\Gamma(n)_d \stackrel{\sim}\to \ds\bigoplus_{\substack{ r > s \geq 2\\ r,s \text{ both even }  \\ r+s =d}}
 \Sigma^{d+1}\Omega(\mathsf{C}_{(1^r)}\circ \mathsf{C}_{(1^s)})(n)
	$$	
\end{lemma}
\begin{proof}  We recall (Proposition $\ref{g2dec}$) that when $d\equiv 2$ mod $4$, $\mathsf{T}^2_\Gamma(n)_d=0$, and so 	$\mathsf{T}^{1,1}_\Gamma(n)_d =	\mathsf{T}_\Gamma(n)_d $.  The chain complex $\mathsf{T}^{1,1}_\Gamma(n)_d$ splits over the labels of the degrees of the special vertices; for a pair of numbers $r\geq s$, denote the associated summand by $\mathsf{T}^{1,1}_\Gamma(n)_{r,s}$.  By Lemma $\ref{g1fa}$, this summand is non-zero when $r$ and $s$ are both even, and since $r+s=d \equiv 2$ mod $4$, we must have $r>s$.  To prove the claim, it thus suffices to construct a quasi-isomorphism of degree $-d-1$, 
$$\eta\colon\mathsf{T}^{1,1}_\Gamma(n)_{r,s} \stackrel{\sim}\to \mathsf{T}(n,r,s)^\ast,$$
where here $\mathsf{T}(n,r,s)^\ast$ is the linear dual of the complex $\mathsf{T}(n,r,s)$ defined above (Definition $\ref{2colordef}$).

To give the construction it will help to first look closer at $\mathsf{T}^{1,1}_\Gamma(n)_{r,s}$.  This chain complex is spanned by $n$-trees with two special vertices, each carrying a label by the appropriate $H^r(\Gamma_{1,a^{-1}(v_1)})$ and $H^s(\Gamma_{1,a^{-1}(v_2)})$.  The set $a^{-1}(v_1)$ has a distinguished element, namely the leg adjacent to $v_1$ in the direction of $v_2$.  In the presence of a such a base point, the cohomology $H^r(\Gamma_{1,a^{-1}(v_1)})$ has a basis given by subsets of $a^{-1}(v_1)$ of size $r$ which don't contain the base point (see \cite[Lemma 3.6]{WardW}).  So we may specify the cohomology class labeling this vertex by marking $r$ legs adjacent to $v_1$, which don't lie on the path connecting $v_1$ and $v_2$.  These markings are indistinguishable and may formally be considered as carrying degree $1$.  Similarly we may mark $s$ of the legs around $v_2$, but not on the path connecting $v_1$ to $v_2$.

To define $\eta$, we first count the number of edges on the unique shortest path connecting the two special vertices on the tree $(\mathsf{t},v_1\,v_2)$.  If there is more than 1 edge, we define the restriction of $\eta$ to this summand to be $0$.  If there is exactly one edge, we contract it to produce a tree with one distinguished vertex, having a set of $r+s$ markings which are further subdivided into two sets of size $r$ and $s$ (by different colors say).  This is exactly the description of the homogeneous elements in (the linear dual of) $\mathsf{T}(n,r,s)$ in Definition $\ref{2colordef}$.

The fact that this assignment is compatible with the differentials is immediate, and so it remains to show that it is indeed a quasi-isomorphism.  For this, filter the mapping cone so that the associated graded permits only expansion of edges which increase the distance between $v_1$ and $v_2$ (measured by the number of edges on the shortest possible path connecting them).  The associated graded splits into summands which can be identified with the augmented chains on a simplicial subdivision of a cube, hence has no homology.
\end{proof}

Finally we compute the homology in the case of $\mathsf{T}_\Gamma(n)_d$ when $4|d$.  First we consider the subcomplex $\mathsf{T}_\Gamma^{1,1}(n)_d$.  The analysis is exactly the same as in Lemma $\ref{mod2}$, except here we must consider the case when $r=s$, both even.  For this we define $\mathsf{C}_{(d/2)}\circ^{S_2} \mathsf{C}_{(d/2)}$ to be the sub-$\mathbf{FA}$-module of $\mathsf{C}_{(d/2)}\circ \mathsf{C}_{(d/2)}$ consisting of summands whose corresponding partitions have only even entries.  
We then define $\mathsf{C}_{(d/2)^\ast}\circ^{S_2} \mathsf{C}_{(d/2)^\ast}$  to be sum of the $\mathbf{FA}$-modules corresponding to the adjoint partitions; explicitly
$$\mathsf{C}_{(d/2)^\ast}\circ^{S_2} \mathsf{C}_{(d/2)^\ast} := \bigoplus_{i=0}^{d/4} \mathsf{C}_{(2^{2i}, 1^{d-4i})}.$$
With this definition we prove:
\begin{lemma}\label{0mod4}
	If $d\equiv 0$ mod $4$ then there exists an $S_n$-equivariant quasi-isomorphism
	$$
 \mathsf{T}^{1,1}_\Gamma(n)_d  \stackrel{\sim}\to\Sigma^{d+1}\Omega(\mathsf{C}_{(d/2)^\ast}\circ^{S_2} \mathsf{C}_{(d/2)^\ast})(n) \oplus \ds\bigoplus_{\substack{ r > s \geq 2\\ r,s \text{ both even }  \\ r+s =d}}\Sigma^{d+1}\Omega(\mathsf{C}_{(r)^\ast}\circ \mathsf{C}_{(s)^\ast}) (n).
	$$	
\end{lemma}
\begin{proof}  Copy the proof of Lemma $\ref{mod2}$ with the additional observation that when $r=s$ the special vertices $v_1$ and $v_2$ are indistinguishable, hence $\mathsf{T}_\Gamma(n)_{d/2,d/2}\sim \Sigma^{d+1}\Omega(\mathsf{C}_{(d/2)^{\ast}}\circ^{S_2} \mathsf{C}_{(d/2)^{\ast}})$.\end{proof}


Combining the previous two lemmas we find:

\begin{corollary}\label{tformula2} 
 Let $d$ be even and let $a\geq b\geq 0$ be integers such that $a+b=d$.  Define $t_{(a,b)}$	by
 	$$	t_{(a,b)} = 
 \begin{cases}
 	\left\lfloor\ds\frac{(a-b)}{4} \right\rfloor + 1 -\delta_{a,b}-\delta_{b,0} & \text{ if } a\equiv b \equiv 0 \text{ mod 2} \\ 
 	& \\ 
 	\left\lfloor\ds\frac{(a-b)}{4} \right\rfloor  -\delta_{b,0} & \text{ if } a\equiv b \equiv 1 \text{ mod 2} \\
 	& \\ 
 	0 & \text{ else.}
 \end{cases}	
 $$	
Then
		$$ 
		H^\ast(\mathsf{T}_\Gamma(n)_d)= \bigoplus_{\substack{ a\geq b \\ a+b=d }} \Sigma^{d+1}t_{(a,b)}\Phi[n,(a,b)^\ast]
		$$ 
\end{corollary}

\begin{proof}  Let us first consider the case that $d\equiv 2$ mod $4$.  From Lemma $\ref{mod2}$, the coefficient $t_{(a,b)}$ is the same as the number of copies of $\mathsf{C}_{(a,b)}$ appearing in the expression	
		$$
		\ds\bigoplus_{\substack{ r > s \geq 2\\ r,s \text{ both even }  \\ r+s =d}}	\mathsf{C}_{(r)}\circ \mathsf{C}_{(s)} =
		\ds\bigoplus_{\substack{ r > s \geq 2\\ r,s \text{ both even }  \\ r+s =d}} \bigoplus_{j=0}^s \mathsf{C}_{(r+j,s-j)}
		$$
		A given $\mathsf{C}_{(a,b)}$ will appear in this expression for each solution to $a=r+j$ and $b=s-j$ arising from $(r,s,j)$ as indexed.  This in turn requires $a+b=r+s=d$ and $r\leq a$.  Since $d$ is even, we must have $a\equiv b$ mod $2$ for such a solution to exist.  When $a\equiv b$ mod $2$, we can set $s=a+b-r$ to find one solution for each even $r$ such that 
		$(a+b)/2 < r\leq a$, except in the case that $b=0$, when we must require $(a+b)/2 < r < a$, since $s$ can't be $0$, 
		accounting for the $\delta_{b,0}$ in the formula.  Counting the number of even integers within this interval gives the result.

	In the case $d\equiv 0$ mod $4$, the analysis is similar, but by Lemma $\ref{0mod4}$,  we also have a contribution when $r=s$, which in turn requires $a$ and $b$ to be even. We appeal to
	 \cite[Section 5.4]{CHKV} to conclude that the connecting homomorphism of the short exact sequence in Equation $\ref{ses}$ is non-zero only if $a=b=r=s$, in which case it sends the cohomology of $\mathsf{T}^{2}_\Gamma(n)_{d} = \Omega(\mathsf{C}_{(d/2,d/2)^\ast})$ to the corresponding summand in $\Omega(\mathsf{C}_{(d/2)^\ast}\circ^{S_2} \mathsf{C}_{(d/2)^\ast})$.  Hence the connecting homomorphism cancels one copy of $\Phi[n,(r,r)^\ast]$, accounting for the $-\delta_{a,b}$ term.

\end{proof}

\begin{proposition} \label{yab} There is an isomorphism of graded $S_n$-modules:
$$H^\ast(\mathsf{T}_\Gamma(n)) \cong \bigoplus_{\substack{0\leq b \leq a \\ a+b\leq n }} y_{(a,b)}\Sigma^{a+b+1}\Phi[n,(a,b)^\ast]$$
where 
$$
y_{(a,b)} = \begin{cases}
	\left\lfloor\ds\frac{a-b+2}{3} \right\rfloor -\delta_{b,0}	&	a \equiv_2 b \\  & \\
	0	&	a \not\equiv_2 b
\end{cases}
$$

\end{proposition}

\begin{proof}  Combining Lemma $\ref{uformula}$ and Corollary $\ref{tformula2}$ we see that  $y_{(a,b)}=u_{(a,b)}+t_{(a,b)}$.  Routine simplification then gives the stated formula. 
\end{proof}

We remark that this result shows the cohomology $H^\ast(\mathsf{T}_{\Gamma}(n))$ is concentrated in degrees $n+1$ and $n$.

\subsubsection{Character computation for $H^\ast_c(F(S^1\vee S^1,n))$}

To prove Statement (1) of the Theorem $\ref{ftthm}$, it remains to compare the decomposition above with the associated decomposition of the configuration space.  For this we let $X =S^1\vee S^1$ and  adapt the arguments in \cite[Theorem 6.1]{GH} to our setting.  We start with
\begin{equation}\label{deceq} 
	H^\ast_c(F(X,n))\cong \bigoplus_m \Phi[n,m]\tensor_{S_m}\tilde{H}^\ast(X)^{\tensor m},
\end{equation} viewed as an equality of $GL_2(\mathbb{Z})\times S_n$ representations 
with $GL_2(\mathbb{Z})$ action is induced from matrix multiplication (diagonally) on $\tilde{H}^\ast(X)^{\tensor m}$.  

We will consider only a restriction of the $GL_2(\mathbb{Z})$ action.
Define matrices $$X = \begin{bmatrix} 1 & 1 \\ -1 & 0 \end{bmatrix} \ \text{ and } \ 
Y = \begin{bmatrix} 0 & 1 \\ 1 & 0
 \end{bmatrix}.$$  
The subgroup $\langle X, Y\rangle \subset GL_2(\mathbb{Z})$ is $D_{12}$, the dihedral group of order 12.  This group is in turn isomorphic to $\mathbb{Z}_2\times S_3$ via $Y\mapsto (0,(12))$ and $X \mapsto (1,(123))$.  Consider the irreducible $D_{12}=\mathbb{Z}_2\times S_3$ representation given by $Id\boxtimes V_{2,1}$.

\begin{lemma}\label{qlambda} There is an isomorphism of graded $S_n$-representations $$H^\ast_c(F(S^1\vee S^1,n))\tensor_{\mathbb{Z}_2\times S_3}( Id\boxtimes V_{2,1})\cong\bigoplus_\lambda \Sigma^{|\lambda|} q_\lambda\Phi[n,\lambda^\ast]$$ where
$$
q_{\lambda} = \begin{cases}
	\left\lfloor \ds\frac{a-b+2}{3} \right\rfloor & \text{ if } \lambda = (a,b) \text{ and } a\equiv b \text{ mod } 2, \\ \\
	0 & \text{ else.}
\end{cases} 
$$
\end{lemma}

\begin{proof}

For this, we take the scalar product of $\mathbb{Z}_2\times S_3$-representations, viewed as a (virtual) $S_n$-representation
$$
\left\langle  Id\boxtimes V_{2,1}, Res^{GL_2(\mathbb{Z})}_{\mathbb{Z}_2\times S_3}(H_c^{\ast}(F(S^1\vee S^1), n)) \right\rangle = 
\sum_\lambda \langle  Id\boxtimes V_{2,1}, Res^{GL_2(\mathbb{Z})}_{\mathbb{Z}_2\times S_3}(\mathbb{S}_\lambda(\mathbb{Q}^2) \rangle \cdot \Phi[n,\lambda^\ast]
$$
where $\mathbb{S}_\lambda(\mathbb{Q}^2)$ denotes the Schur functor associated to the partition $\lambda$. 
As in \cite[Section 6.1]{GH}, it suffices to compute $$q_\lambda:=\left\langle  Id\boxtimes V_{2,1},Res^{GL_2(\mathbb{Z})}_{\mathbb{Z}_2\times S_3}(\mathbb{S}_{\lambda}(\mathbb{Q}^2))\right\rangle $$
Note $q_\lambda=0$ unless $\lambda=(a,b)$, which we assume from this point.  
The potential non-zero contributions come from the remaining conjugacy classes, whose data is summarized as follows:

\begin{center}
	\begin{tabular}{c|c|c|c|c|c}
		Conj. Class & size of & Matrix & Eigenvalues & $\chi( Id\boxtimes V_{2,1})$ & $\chi(Res (\mathbb{S}_\lambda(\mathbb{Q}^2)))$ \\ \hline 
		&&&&& \\
		$id \times id$ & 1 & $\begin{bmatrix} 1 & 0 \\ 0 & 1\end{bmatrix}$ & $1,1$ & 2 & $a-b+1 $   \\ 
		&&&&& \\
		$id \times (123)$ & 2 & $\begin{bmatrix}-1 & -1 \\ 1 & 0\end{bmatrix}$ & $\zeta_3,\zeta_3^2$& -1 & $(a-b+1)$ mod 3 	\\
	&&&&& \\	
		$(12) \times (123)$ & 2 & $\begin{bmatrix} 1 & 1 \\ -1 & 0 \end{bmatrix}$ & $\zeta_6,\zeta_6^{-1}$& -1 & $(-1)^{a+b}(a-b+1)$ mod 3 	\\
			&&&&& \\			
		$(12) \times id$ & 1 & $\begin{bmatrix}-1 & 0 \\ 0 & -1\end{bmatrix}$ & $-1,-1$& 2 & $(-1)^{a+b} (a-b+1)$	\\
	\end{tabular}	
\end{center}

The character of $\mathbb{S}_{\lambda}(\mathbb{Q}^2)$ is determined, as in \cite{GH}, via the eigenvalues; a matrix with eigenvalues $x$ and $y$ has character $\sum_{i=b}^a x^iy^{a+b-i}$.  The last column takes residues valued in $\{-1,0,1\}$.  We thus see that the scalar product will be $0$ unless $a \equiv b$ mod $2$, in which case it's  equal to $$\ds\frac{1}{12}\left(
4(a-b+1) -4(a-b+1 \text{ mod } 3)    \right). $$

Simplifying this expression we find   
$$
\frac{(a-b+1)}{3} -\frac{(a-b+1) \text{ mod } 3}{3} = \left\lfloor \ds\frac{a-b+2}{3} \right\rfloor$$
as desired.
\end{proof}




\subsubsection{Finishing the Proof}

Comparing $q_{(a,b)}$ in Lemma $\ref{qlambda}$ with $y_{(a,b)}$ in Proposition $\ref{yab}$, we see they coincide except when $a$ is even and $b=0$, in which case $y_{a,b} = \lfloor \frac{a-1}{3} \rfloor$ and $q_{(a,b)} = \lfloor \frac{a+2}{3} \rfloor$, i.e.\ $q_{(a)} = y_{(a)}+1$.  Summing over all such $a$ we find
$$
\Sigma^{-1}H^\ast(\mathsf{T}_\Gamma(n))\oplus (\oplus_i\Sigma^{2i}\Phi[n,(2i)^\ast])\cong H^\ast_c(F(S^1\vee S^1,n))\tensor_{\mathbb{Z}_2\times S_3} (Id\boxtimes V_{2,1})
$$
Recalling $\mathsf{W}_{\pm}:=\oplus_i\Sigma^{2i}\Phi[n,(2i)^\ast]$ (Equation $\ref{Wdef}$), and verifying $V_{120} = Id\boxtimes V_{2,1}$ concludes the proof of Statement (1) in Theorem $\ref{ftthm}$.

\section{$\mathbf{FA}$-decomposition of $B(2,n,r)$.}

In this section we derive a decomposition of a graph complex denoted $B(g,n,r)$ in terms of the cobar construction of $\mathbf{FA}$-modules.  This graph complex was introduced \cite{PW} in the case $r=11$ to model the weight 11 component of the compactly supported cohomology of the moduli space $\op{M}_{g,n}$.  In the case $g=1$, the cohomology of $B(1,n,r)$ was computed in \cite{FW} in terms of Whitehouse modules. 
The theorem we prove here can be thought of as a generalization of this result to genus 2.

 With enough set-up, the graph complex $B(g,n,r)$ can be efficiently described via the Feynman transform, see Remark $\ref{BFTrmk}$, but we find it more convenient for our purposes to simply give a hands-on description, which we do now.

\subsection{Recollection of $B(g,n,r)$}
Recall that basic notation and terminology related to graphs was given in sub-section $\ref{graphs}$.  We now specialize to what we call marked graphs.  By a marked graph $(\gamma, v, R)$ we refer to a connected, leg-labeled graph along with a choice of a distinguished vertex $v$ and a subset $R\subset a^{-1}(v)$.  We call the elements of $R$ marked flags, or markings, and denote them graphically by tic marks (see the top line of Figure $\ref{fig:fig1}$).  We impose the stability condition $a^{-1}(w) \geq 3$ for all vertices $w\in V(\gamma)\setminus v$, which we call neutral vertices, but impose no stability condition at $v$ itself.  We also impose the condition that marked graphs have no tadpoles adjacent to neutral vertices.  We say a marked graph is of type $(g,n,r)$ if $\beta(\gamma) = g-1$, $\gamma$ has $n$ legs and  $|R|\geq r$.  We remark that the $|R|\geq r$ condition means that one can't determine $r$ from looking at the graph alone, rather we first fix the parameters $(g,n,r)$ and then consider the associated marked graphs.

Given a marked graph we write $det(\gamma)$ for the top exterior power of the span of the set $E(\gamma)$.  We write $det^{-1}(R)$ for the top exterior of the set $R$, shifted down in degree by $2R$, so-as to be concentrated in degree $-|R|$.  We then define $det(\gamma, R):= det(\gamma)\tensor det^{-1}(R)$, a 1-dimensional vector space concentrated in degree $|E(\gamma)|-|R|$.  The graded vector space underlying the chain complex $B(g,n,r)$ is then defined to be $$B(g,n,r) = \ds\bigoplus_{[\gamma,v,R]} det(\gamma, R)_{Aut(\gamma,R)}.$$
Here morphisms, and hence automorphisms, of marked graphs are presumed to preserve the data $v$ and $R$.  

The differential on $B(g,n,r)$ can be described informally as the sum over all edge expansions and ways to erase a marking on a marked flag (i.e.\ remove an element of $R$ provided $|R|>r$).  The edge expansions at the distinguished vertex are non-zero only if at most one marked flag is being expanded away from the distinguished vertex.  If there is one such marked flag it marks the expanded edge.

We refer to \cite{PW} and \cite{FW} for a greater level of detail.  To make a direct comparison, note that the complex $\tilde{B}(g,n)$ of \cite{PW} coincides with our $B(g,n,11)$. 

\begin{remark} \label{BFTrmk} It is possible to describe the graph complex $B(g,n,r)$ directly in terms of the Feynman transform.  For this, one should first resolve the $\mathbf{FA}$-module $\widetilde{\mathsf{C}}_{1^r}$ via the truncation of the Koszul complex $\dots \to\mathsf{C}_{1^{r+1}} \to \mathsf{C}_{1^{r}} \to 0 $ and then form a modular cooperad from this data as in \cite{CLPW2}.    
\end{remark}

\subsection{Main Theorem}\label{secmain}

We now give the main theorem of this section.  We recall the notation $\lambda^\ast$ indicates the partition conjugate to $\lambda$.

\begin{theorem}\label{decompthm}  For all $n$ and $r\geq 3$, there is an $S_n$-equivariant quasi-isomorphism

		
		
					$$\Sigma^{r-2}B(2,n,r)  \sim
		\Omega(\mathsf{C}_{(r-2,1)^\ast})(n) \oplus \ds\bigoplus_{i=1}^{\lfloor(n-r+1)/2\rfloor}\Sigma^{2i}\Omega(\mathsf{C}_{(2i+1)^\ast}\circ\mathsf{C}_{(r-2)^{\ast}} )(n)  $$	
\end{theorem}

We remark that the shifts in the statement imply that possible degrees of the homology of each summand align.  In particular, an immediate corollary is:

\begin{corollary}
	$H^j(B(2,n,r)) = 0$ unless $j=n-2r+3$ or $j=n-2r+2$.
\end{corollary}

We have stated Theorem $\ref{decompthm}$ when $r \geq 3$.  We remark that the theorem also holds when $r=2$, provided that the term corresponding to $(2,1^{r-3}) = (r-2,1)^\ast$ in the statement is interpreted as $0$, and the term  $\mathsf{C}_{(2i+1)^\ast}\circ\mathsf{C}_{(r-2)^{\ast}}$ is interpreted as $\mathsf{C}_{(2i+1)^\ast}$.  When $r=0,1$ the cohomology of $B(2,n,r)$ is essentially trivial, 
see Corollary $\ref{lowr}$.

The proof of Theorem $\ref{decompthm}$ will occupy the rest of this subsection. 
We find it slightly more natural to give the proof of the linear dual statement, in which the cobar construction is replaced by the bar construction.  For this we define $L(n,r):=B(2,n,r)^\ast$.  The proof will proceed in four steps:
\begin{enumerate}
	\item  Identify a class of genus 2 graphs forming an acyclic subcomplex of $L(n,r)$ (we call these ``class-{\ttfamily b}'').   Reduce to the associated quotient complex of ``class-{\ttfamily a}'' graphs, which we denote by $\mathsf{L}(n,r)$.

\item Construct a complex of trees, the total tree complex $\mathsf{TT}(n,r)$, which is a slight variation and quasi-isomorphic to the (linear dual of the) right hand side of Theorem $\ref{decompthm}$.

\item Define a chain map $\eta\colon\mathsf{L}(n,r) \to \mathsf{TT}(n,r)$, roughly by contracting even length cycles.

\item Verify that $\eta$ is a quasi-isomorphism by filtering the mapping cone.  The associated graded splits into summands which can be compared to chains on permutohedra.
\end{enumerate}

\subsubsection{Reduction to ``type-a'' graphs} 

As above $L(n,r):=B(2,n,r)^\ast$.  After identifying the invariants of $Aut(\gamma)$ with the coinvariants (via the isomorphism $x\mapsto [x]$), we may consider $L(n,r)$ as a span of isomorphism classes of marked graphs of type $(2,n,r)$, modulo the action of automorphisms, with differential given as a sum of (non-tadpole) edge contractions and ways to mark an unmarked flag at the distinguished vertex (DV for short).  When an edge with a marked flag (also called a ``marked edge'') is contracted, the marking is distributed over the flags newly adjacent to the new DV. 
 Our degree conventions for $L(n,r)$ are homological with edges still having degree $1$ and markings degree $-1$.

Before proceeding, let us fix some terminology related to the graphs appearing in the argument.  First, note that a marked graph $\gamma$ of type $(2,n,r)$ necessarily has $\beta(\gamma)=1$, hence a unique cycle.  This in turn partitions the sets of vertices and edges into those which lie ``off'' and ``on'' the cycle respectively.  

Second, if we were to erase the edges on the cycle we'd be left with a disconnected graph whose components are rooted trees, with root vertices given by the cycle vertices.   We call these components the ``branches'' of the graph. We call a branch ``distinguished'' if the DV belongs to it, else it is ``neutral''. We use the terminology ``input'' and ``output'' for flags on a branch, viewing the branch as oriented toward its root/cycle vertex.  In particular, every flag not on the cycle is classified as either an input or an output.  Flags on the cycle are classified as neither.  See Figure $\ref{fig:f2}$.

\begin{definition}  A marked graph of type $(2,n,r)$ is defined to be ``class-{\ttfamily a}'' if all of the following conditions are satisfied:
	
	\begin{enumerate}
		\item[({\ttfamily a1})] the distinguished vertex (DV) lies on the cycle,
		\item[({\ttfamily a2})] the number of marked flags $|R_\gamma|$ is exactly $r$,
		\item[({\ttfamily a3})] both (or in the case of a tadpole, the) adjacent cycle edges are marked.
	\end{enumerate}
Such a marked graph is defined to be of ``class-{\ttfamily b}'' if it's not of class-{\ttfamily a}, i.e.\ if one or more of the above conditions is not satisfied.
\end{definition}

We invoke the above definition for all $r\geq 0$.  When $r=0$, the conditions {\ttfamily (a1),(a2)} and {\ttfamily (a3)} can not be simultaneously satisfied, so there are no class-{\ttfamily a} graphs in that case.  When $r=1$, the conditions can be satisfied only if the cycle has one edge (so is a tadpole).

\begin{lemma}\label{bgraphs}  The span of class-{\ttfamily b} graphs is an acyclic subcomplex of $L(n,r)$.
\end{lemma}
\begin{proof} 
	First, we confirm class-{\ttfamily b} graphs indeed form a subcomplex.  A graph is class-{\ttfamily b} if one or more of the following hold:

	\begin{enumerate}
		\item[(\ttfamily{b1})] the DV is not on the cycle
		\item[(\ttfamily{b2})] the number of markings is more than $r$
		\item[(\ttfamily{b3})] the DV is on the cycle, but an adjacent cycle edge is not marked.
	\end{enumerate}

Differential terms which add a marking will always satisfy ({\ttfamily b2}).  Differential terms which contract an edge preserve the number of markings, hence preserve property ({\ttfamily b2}).  It thus remains to examine the impact of contracting an edge of a graph satisfying properties ({\ttfamily b1}) or ({\ttfamily b3}).  If a marked graph satisfies condition ({\ttfamily b1}), contracting an edge not connecting the cycle to the DV will satisfy ({\ttfamily b1}), and contracting an edge which connects the cycle to the DV (if such an edge exists) will satisfy ({\ttfamily b3}).   In this last case, note that the cycle must have length $\geq 2$ (since there are no neutral tadpoles) hence the differential would be distributing one marking over two distinct edges.  

We thus reduce to considering the case of edge contraction for a graph satisfying property  ({\ttfamily b3}).  Clearly, contraction of an edge not on the cycle preserves property ({\ttfamily b3}).  Also, if the cycle has length $\neq 2$, contracting a cycle edge preserves property ({\ttfamily b3}) -- note here contracting a tadpole is 0 by definition.  In the case where the cycle has length 2, if both edges are unmarked the graph vanishes in $L(n,r)$ due to the automorphism exchanging the edges.  The last remaining case is a cycle of length $2$ with one marked edge.  In this case, contracting the unmarked edge leaves a marked tadpole, while contracting the marked edge has one term corresponding to the same marked tadpole, distributing the marking to the formerly unmarked edge.  We thus see a type-{\ttfamily a} graph (the marked tadpole) but it appears twice with opposite sign.  Hence type-{\ttfamily b} graphs form a subcomplex.

We then proceed to prove this subcomplex is acyclic.  For this we will introduce a filtration defined in terms of the following statistics associated to a marked graph $\gamma$ of type $(2,n,r)$:
\begin{itemize}
	\item $c=c(\gamma)$ is the length (i.e.\ the number of edges)  of the unique cycle of $\gamma$,
	\item $e=e(\gamma)$ is the number of edges of $\gamma$,
	\item $I= I(\gamma)$ is the number of marked flags of $\gamma$ which are inputs,
	\item $\nu=\nu(\gamma)$ is defined to be $0$ unless the DV is off the cycle, the output of the DV is marked, the number of markings is exactly $r$, and the root vertex of the distinguished branch has a unique input flag. If all such conditions are satisfied, define $\nu(\gamma)=-1$.   
\end{itemize} 	
See Figure $\ref{fig:f2}$.

\begin{figure}
	\centering
	\includegraphics[width=0.8\linewidth]{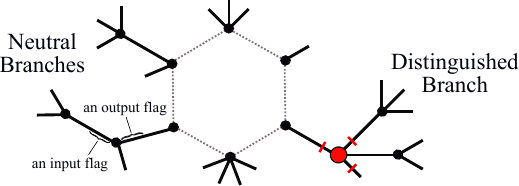}
	\caption{A class-{\ttfamily b} graph in $B(2,21,3)$ with $|R|=3$, $c=6$, $I=2$ and $\nu=-1$.  The distinguished vertex (DV) in red has marked output.  The cycle edges are drawn as dotted lines.  The leg labeling is suppressed.}
	\label{fig:f2}
\end{figure}

We then define a filtration $q$ on the subcomplex of class-{\ttfamily b} graphs by $$q(\gamma) = e+c+\nu-2I.$$  
To see that $q$ is indeed a filtration, note that $q$ can not increase when considering the terms in the differential.  More precisely, let us consider which terms in the differential fix versus decrease the filtration degree.

First, regarding edge contraction: $q$ decreases when contracting a cycle edge or an edge on a neutral branch; $q$ also decreases when contracting an edge on the distinguished branch unless $\nu$ increases.  There are two ways $\nu$ can increase: either we contract the (necessarily unique) edge of the distinguished branch which is adjacent to the root, or we contract the output edge of the DV and distribute the output marking to be an input.  In this latter case $I$ also increases, hence $q$ decreases.  So the only edge contraction which doesn't lower the filtration degree is contraction of the unique edge adjacent to the cycle and on the distinguished branch of $\gamma$, provided $\nu(\gamma)=-1$.

Second, regarding marking flags at the DV, the filtration degree $q$ decreases when marking an input flag and is fixed by marking non-input flags.  If the DV is off the cycle there is a unique such flag (the output).  If the DV is on the cycle there are (at most) two such flags (corresponding to the adjacent cycle edges(s)).

We aim to show the associated graded of this filtration has no homology.  By the above analysis, the differential terms in the associated graded are quite sparse, 
and we use this to argue that the associated graded splits into acyclic summands as follows:
\begin{itemize}
	\item If the DV is off the cycle and its output is unmarked, we mark it, resulting in marked output and more than $r$ markings.  Such a pair receives or supports no other differential terms in the associated graded, so supports an acyclic summand. 
	\item If $\nu(\gamma)=-1$ and the path connecting the DV to the cycle has more than one edge, 
	we contract the unique branch edge adjacent to the root.  This results in a graph $\gamma^\prime$ for which $\nu(\gamma^\prime)=0$, with a DV which is off the cycle, having marked output and $r$ markings. 
	This pair also supports an acyclic summand in the associated graded.
	\item Suppose $\nu(\gamma)=-1$ and the path connecting the DV to the cycle has only one edge, which in turn requires $c(\gamma)\geq 3$.  Contracting said edge gives a sum of two terms in the associated graded, distributing the output marking to either cycle edge.  Both differential terms have exactly $r-1$ marked input flags.  Each of these terms is itself the source of a differential term in the associated graded which marks the opposite, unmarked cycle edge.  The result is an acyclic summand, coinciding with the augmented simplicial chain complex of an interval. When $c=2$, automorphisms reduce the analogous summand to two terms (one or both edges marked, $r-1$ input markings) but the summand is still acyclic.
	
	\item Finally, if the DV is on the cycle and the number of markings off the cycle is at least $r$, then such a graph  belongs to a summand of the associated graded consisting of all four configurations of the cycle edges marked and unmarked, (unless $c\leq 2$ in which case there are only two configurations).  In either case the resulting summand is the augmented simplicial chain complex of an interval (or when $c\leq 2$, a point) and hence is acyclic.
	\end{itemize}
We emphasize that each class-{\ttfamily b} graph belongs to exactly one of the above summands, so this is indeed a splitting of this subcomplex into acyclic summands.  Since the associated graded has no homology, neither does the subcomplex of class-{\ttfamily b} graphs.
\end{proof}

We thus look to compute the homology of the quotient complex of class-{\ttfamily a} graphs.  We give this complex a name:
\begin{definition} Define $\mathsf{L}(n,r)$ to be the quotient complex of $L(n,r)$  by the subcomplex of class-{\ttfamily b} graphs.
\end{definition}

Before moving to the general case, let us use Lemma $\ref{bgraphs}$ to handle the following slightly simpler, borderline cases.
\begin{corollary}\label{lowr}
	The complex $B(2,n,r)$ has no cohomology when $r=0,1$ except $B(2,1,1)$ and $B(2,0,1)$, each of which have cohomology of rank 1 in degree $0$.

\end{corollary}

\begin{proof}
		When $r=0$ there are no class-{\ttfamily a} graphs and so $\mathsf{L}(n,0)$ (hence $B(2,n,0)$) is acyclic.
	
	When $r=1$ and $n\geq 2$ the complex $\mathsf{L}(n,1)$ always has a marked tadpole (which can't be double marked).  We can always expand/contract an edge separating this marked tadpole from adjacent input flags to show $\mathsf{L}(n,1)$ (hence $B(2,n,1)$) is acyclic.
	
	When $r=1$ and $n=0$ or $1$, $\mathsf{L}(n,1)$ consists of a marked tadpole and no other graph, hence the claim.
\end{proof}

\subsubsection{Forming the total complex of trees.}  

We now impose the condition $r \geq 2$, the cases $r=0,1$ having been treated above.

Recall (Definition $\ref{2colordef}$) that $\mathsf{T}(n,b,r)$ denotes the chain complex
	$$\mathsf{T}(n,b,r) := 
\mathsf{B}(\mathsf{C}_{(b)^\ast}\circ\mathsf{C}_{(r)^\ast})(n).$$
This chain complex may be described combinatorially as pointed $n$-trees along with a choice of $b+r$ markings which are partitioned into an ordered pair of subsets of size $b$ and $r$.  We call the set of size $b$ the ``blue markings'' and the set of size $r$ the ``red markings''.  The degree in this chain complex is given by the number of edges and ranges from $0$ to $n-(b+r)$.  Note that when $b=r$ we still retain the information of the order of the subsets; so for example $\mathsf{T}(2,1,1)$ has dimension 2.

\begin{definition}\label{ttdef} Given $n$ and $r$ we define the chain complex $(\mathsf{TT}(n,r), d_{\mathsf{TT}})$ as follows.  The underlying graded vector space is

$$ \mathsf{TT}(n,r) = \left(\ds\bigoplus_{i=0}^{\lfloor (n-r+1)/2\rfloor} \Sigma^{2i-r+2}\mathsf{T}(n,2i+1,r-2)\right)\oplus \Sigma^{1-r}\mathsf{T}(n,0,r-1).$$	
The differential $d_{\mathsf{TT}}$ is of the form $d_{\mathsf{TT}}= d_{\mathsf{T}} + (-1)^{E} d_R$ 
where $d_\mathsf{T}$ is the differential on each summand, and where $d_R\colon \mathsf{T}(n,1,r-2) \to \mathsf{T}(n,0,r-1)$ is defined to take the unique blue marking and turn it red. 
The fact that $d_{\mathsf{TT}}^2 = 0$ follows immediately from the fact that edge contraction and changing the color of a marking commute; the sign ensures they anti-commute. 
\end{definition}
We call $\mathsf{TT}(n,r)$ the ``total tree'' complex.  To prove Theorem $\ref{decompthm}$, it's now sufficient to construct a quasi-isomorphism, call it $\eta$,
$$
\mathsf{L}(n,r) \stackrel{\eta}\to \mathsf{TT}(n,r).
$$
Indeed the terms with $i\geq 1$ in Definition $\ref{ttdef}$ match those in the direct sum in Theorem $\ref{decompthm}$.  The remaining $\Omega(\mathsf{C}_{(r-2,1)^\ast})$ term in the theorem arises from the homology of $d_R\colon \mathsf{T}(n,1,r-2) \to \mathsf{T}(n,0,r-1)$.  This includes the case of $r=2$, where $d_R$ is an isomorphism. 

\subsubsection{Defining the map by even cycle contraction}
We now define this putative quasi-isomorphism.  Informally, $\eta$ is defined by ``contracting'' even length cycles and tadpoles, while leaving behind a new color (blue say) of marking, one for each neutral vertex that was contracted.

To give the formal definition we define $\eta$ on the summand $det(\gamma,R)_{Aut(\gamma,R)}$ associated to a class-{\ttfamily a} marked graph $\gamma=(\gamma,v,R)$ of type $(2,n,r)$, and then extend linearly. 
Write $\eta_\gamma$ for the restriction to the $\gamma$ summand.  Recall $c=c(\gamma)$ is the length of the unique cycle in $\gamma$.  If $c$ is odd and not equal to 1, we define $\eta_\gamma = 0$.  If $c=1$, i.e. $\gamma$ has a marked tadpole, then $\eta_\gamma$ is defined to remove the tadpole and the marking to form a tree supporting a summand in $\mathsf{T}(n,0,r-1)$, and hence in $\mathsf{TT}(n,r)$.  Since the edges and markings of this tree are in canonical bijective correspondence with the edges and markings of $\gamma\setminus \{\text{tadpole}\}$, these summands are isomorphic.  We choose the convention that $\gamma_{\eta}$ is the isomorphism which removes the tadpole edge in the last position and removes its marking in the first position of the wedge product $det(\gamma,R)$.

Finally, suppose $c=2i+2$. Enumerate the neutral vertices $v_1\cdc v_{2i+1}$ in either order around the cycle, starting from the distinguished vertex (call it $v_0$, say).  Define $G_i$ to be the input flags adjacent to $v_i$.  Define $G= G_1\times ... \times G_{2i+1}$. We will define $\gamma_\eta$ to be a sum indexed over the elements of $G$, which will correspond to the new (blue) markings. 

To make this precise, define a linear map 
$$\ell \colon \text{span}\{G\}\tensor det(\gamma,R) \to \mathsf{T}(n,c-1,r-2)\subset \mathsf{TT}(n,r)$$ 
as follows.  
 Given $\vec{g}=(g_1\cdc g_{2i+1})\in G$, we form a pointed $n$-tree, call it $\mathsf{t}_\gamma$ by contracting the cycle, forming a new (distinguished) vertex.  Specify two subsets of the flags at the distinguished vertex as follows.  The first subset (the blue markings) are the elements $\vec{g}$ (which clearly correspond to flags of $\mathsf{t}_{\gamma,\vec{g}}$).  The second subset (the red markings) are the flags of $\gamma$ that were marked but not on the cycle (see the left hand side of Figure $\ref{fig:fig1}$).  This tells us the summand of $\mathsf{T}(n,2i+1,r-2)$ on which $\ell(\vec{g},-)$ is supported.  It thus suffices to fix the signs which define the specific map
$$det(\gamma,R) \to det(\mathsf{t}_{\gamma,\vec{g}}).$$

For this, choose a representative of a basis element of the source for which the edges of the cycle are in last position.  Place the marked edges in the last two positions (in either order) and their corresponding markings in the first two marking positions (mimic the chosen edge order).  For the remaining edges we traverse the cycle in the same order as that of the vertices chosen above.  The order of the other (red) markings is immaterial.  Define the map to take this basis element to that in the target given by keeping the remaining red markings and non-cycle edges in the given order, and taking the blue markings in the chosen order $g_1\wedge ... \wedge g_{2i+1}$.  Observe this definition is independent of choice -- if we had chosen the other orientation of the cycle both the source and target would differ by $(-1)^i$, hence the map would be the same.

Finally, define $\eta_\gamma = \sum_{\vec{g}\in G} \ell(\vec{g},-)$;  in words, we sum over ways to assign one blue marking to an input flag at each neutral vertex, and then contract the cycle to form a tree with a distinguished vertex having two distinguished subsets of markings.

\begin{lemma}  $d_{\mathsf{TT}}\eta=\eta d_{\mathsf{L}}$.
\end{lemma}
\begin{proof}
It's enough to prove the result holds when restricting the input to a summand indexed by a graph $\gamma$.  We perform a case analysis based on the cycle length $c$.

{\bf Case 1:} $c$ is even and $c>2$. In the quotient complex $\mathsf{L}(n,r)$, the differential terms which add a marking vanish, so we need only consider edge contractions.  Contracting an edge of $\gamma$ which lies on the cycle, then applying $\eta$ gives $0$.  The edges off the cycle index the differential terms in $d_{\mathsf{TT}}\eta$.  This comparison requires permuting the non-cycle edge past the cycle edges, which is an even permutation. It remains to observe that contracting a non-cycle edge which is adjacent to a neutral cycle vertex introduces new adjacent flags which index terms in $\eta$ as recipients for the blue markings.  On the other hand, contraction of this edge after applying $\eta$ distributes the marking over all possible newly adjacent flags.  See Figure $\ref{fig:fig1}$.

\begin{figure}
	\centering
	\includegraphics[scale=1.25]{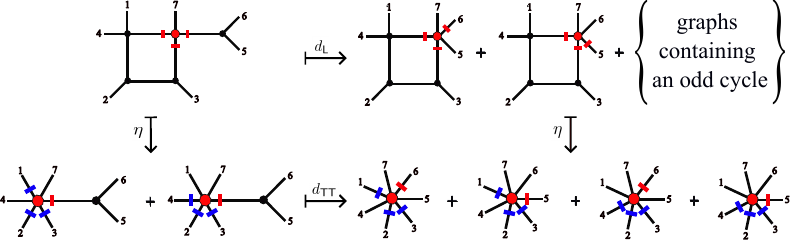}
	\caption{An example of $d_{\mathsf{TT}}\eta=\eta d_{\mathsf{L}}$.  }
	\label{fig:fig1}
\end{figure}

{\bf Case 2:} $c=2$.  The analysis is similar to the previous case, except contracting along the cycle via $d_\mathsf{L}$ doesn't vanish.  If we contract one of the parallel edges via $d_\mathsf{L}$, we distribute the (red) marking, and then can still apply $\eta$ to remove the tadpole.  On the other hand, if we first contract the parallel edges via $\eta$, we sum over choices for the unique blue marking, then apply $d_{R}$ to change the color to red.

{\bf Case 3:} $c=1$.
Terms in $d_{\mathsf{TT}}\eta$ and $\eta d_\mathsf{L}$ both  correspond to contractions of edges of $\mathsf{t}_\gamma$.

{\bf Case 4:} $c$ is odd and $c>1$.   In this case $d_\mathsf{TT}\eta(\gamma) = 0$, so we want to show that $\eta d_\mathsf{L}=0$.  Any terms appearing in $d_\mathsf{L}(\gamma)$ which don't correspond to contracting a cycle edge are mapped to zero by the definition of $\eta$.  So it remains to consider the sum of contractions of edges around the cycle.

As above, we denote the DV as $v_0$ and we enumerate the neutral vertices $v_1\cdc v_{c-1}$, in either orientation, traversing the cycle starting from $v_0$.  For $1\leq i \leq c-1$ we continue with the notation $G_i$ for the input flags adjacent to $v_i$.  In particular $|G_i| = a^{-1}(v_i)-2$.

For $1\leq i\leq c$ we define $e_i$ be the edge which connects $v_{i-1}$ to $v_i$ (mod $c$). Define $d_i$ to be the differential term which contracts $e_i$.  To keep track of the potential recipients for blue markings in each such term, we introduce the following notation:
\begin{itemize}
\item $G[1] = G_2 \times ...  \times G_{c-1}$,

\item $G[i] = G_1 \times ... \times (G_{i-1}\sqcup G_{i})\times G_{i+1}\times ... \times G_{c-1}$ for $1<i<c$,

\item $G[c] = G_1 \times ...  \times G_{c-2}$,

\item $G[\bullet] = \ds\coprod_{i=1}^c G[i]$
\end{itemize}
Notice the set $G[\bullet]$ has a natural involution since each element appearing in $G[i]$ misses either $G_{i-1}$ or $G_i$, and hence the same list appears in either $G[i-1]$ or $G[i+1]$.  For the end points, each list appearing in $G[1]$ appears in $G[2]$ and each list appearing in $G[c]$ appears in $G[c-1]$.  Define this involution (taking one instance of a given list to the other) by $\iota$.

We want to confirm that $\sum_i \eta d_i(\gamma)=0$.  For this, let us abuse notation by identifying the graph $\gamma$ with an element in $\text{det}(\gamma)$, fixing an order of the edges/markings of $\gamma$.  We then compute
$$
\sum_i \eta d_i(\gamma) = \sum_i \eta_{\gamma/e_i} (\gamma/e_i) = \sum_i \ds\sum_{g\in G[i]} \ell(\vec{g},\gamma/e_i).
$$
To conclude, we observe that 
$
\ell(\vec{g},\gamma/e_i) + \ell(\iota\vec{g},\gamma/e_{i\pm 1}) =0,
$
simply because the terms arose from contraction of adjacent edges in either order.  
\end{proof}

\subsubsection{Filtration of the mapping cone.}

To complete the proof it suffices to show the following complex (the mapping cone of the morphism $\eta$) is acyclic.   

\begin{definition}  Let $r\geq 2$.  
Define a chain complex $\mathsf{A}(n,r)$ as follows.  The underlying graded vector space is $$\mathsf{A}(n,r)= \mathsf{L}(n,r)\oplus  \Sigma^{-1}\mathsf{TT}(n,r)$$		

	The differential $d_\mathsf{A}$ has the form $d_{\mathsf{L}}+d_{\mathsf{TT}}+\eta$ where the first two are the component-wise differentials on the direct sum and where $\eta$ contracts even cycles, as above. \end{definition}

We introduce a filtration, call it $\zeta$, of the complex $\mathsf{A}(n,r)$ designed to isolate two types of differential terms -- those which contract neutral cycle edges and those which collapse even length cycles (or remove a tadpole) to produce a tree, provided that each neutral cycle vertex of the original graph was trivalent.

To define the filtration, we start with the case of $\gamma$, a marked graph of type $(2,n,r)$.  We let $e$ be the number of edges, let $w$ be the (valence of) the DV, and let $v_i$ for $i=1,...,c-1$ be the (valence of the) neutral cycle vertices.  By abuse of notation we let $\zeta(\gamma)$ denote the filtration degree of the summand supported by $\gamma$.   We define this filtration degree by:
$$\zeta(\gamma) = 2w-e-c+\sum_{i=1}^{c-1}v_i.$$
This defines $\zeta$ on the summand $\mathsf{L}(n,r)$.  Next we define $\zeta$ on the summand $\mathsf{TT}(n,r)$.  For this we let $\mathsf{t}$ be a pointed $n$-tree along with a choice of $b$ blue and $r$ red marked flags at the DV.  Let $E$ be the number of edges of the tree and let $W$ be the valence of the distinguished vertex.  We again abuse notation by letting $\zeta(\mathsf{t})$ denote the filtration degree of the summand supported by $\mathsf{t}$, and define:
$$\zeta(\mathsf{t}) = 2W-E-b+2$$

Let us verify both that $\zeta$ is indeed an increasing filtration (the filtration degree can't decrease upon applying differential) while also considering which differential terms do or do not fix the filtration degree.

\begin{itemize}
\item Upon contraction of a neutral cycle edge, the sum $\sum_i v_i$ decreases by 2, $c$ decreases by $1$, $e$ decreases by $1$ and $w$ is fixed.  Thus the filtration degree is fixed. 
\item  Upon contraction of a marked cycle edge, $e$ and $c$ each decrease by 1, the sum $w+\Sigma_iv_i$ decreases by $2$, but $w$ itself must increase, so the filtration degree increases. 
\item Upon contracting an edge not on the cycle, $e$ decreases, $c$ is fixed and $2w+\sum_iv_i$ can not decrease.  Thus the filtration degree must increase. 

\item Regarding $d_{\mathsf{TT}}$, upon contracting a tree edge $E$ must go down and $2W-b+2$ can't decrease, hence the filtration degree must increase.  The $d_R$ piece of the differential decreases the number of blue markings, so also increases the filtration degree.

\item If $c(\gamma)=1$, the differential term coming from $\eta$ removes the tadpole. The filtration degree of the source is $2w-e-1$ and of the target is  $2(w-2) - (e-1)-0+2$, so is preserved.  

\item Finally, suppose that $\gamma$ has an even length cycle of length $c\geq 2$.  Upon contraction (application of the map $\eta$ above) we get a sum of trees with $E=e-c$ edges, $b=c-1$ blue markings and with a distinguished vertex of valence $W$ = $w+\sum_iv_i - 2c$.  Thus, the filtration degree of each such tree is  $\zeta(\mathsf{t}) 
= 2w + 2\sum_iv_i - 4c - e+3.$  
Hence, 
$$\zeta(\mathsf{t}) - \zeta(\mathsf{\gamma}) = \sum_{i=1}^{c-1} v_i -3c+3 \geq 0.$$
The last inequality holds due to stability at the neutral vertices.  We conclude that the filtration degree increases unless each neutral cycle vertex is trivalent, in which case the filtration degree is preserved.
\end{itemize}
Generate an equivalence relation on the set of graphs appearing in $\mathsf{A}(n,r)$ by saying that two graphs are equivalent if one appears as a differential term of the other in the associated graded of $\zeta$.   Define the ``order'' of an equivalence class to be the longest cycle length of any graph belonging to it.  By construction, the associated graded splits over such equivalence classes.  The isomorphism type (and hence the homology) of a summand in this splitting depends only on the order of the equivalence class.  

To see this, we choose a standard representative for each order $c$.  Given $c$, consider the polygon in $\mathsf{L}(c,3)$ having $c$ trivalent vertices, having the leg adjacent to the DV labeled by the number $c$ and with neutral vertices adjacent to the legs $1,...,n-1$ in order around the cycle.  Given another equivalence class of order $c$, we choose a graph $\gamma$ which  represents this class and has a cycle of length $c$.  We then choose an order on the set of $c-1$ neutral branches of $\gamma$.  The isomorphism from this summand to the standard summand simply replaces the branch $i$ with a leg labeled by $i$, and replaces the distinguished branch by the leg labeled by $n$.

The summand associated to such a graph has a basic combinatorial description, closely related to the coinvariants of the cellular chains on the Permutohedron, denote $\Pi(c-1)$ by the action of reflecting the word labeling a cell.  We verify this complex is acyclic in the following section.

\subsubsection{Permutohedron Lemma}

Let's prove the core combinatorial argument that will be needed to describe the complex of neutral edge contractions around the cycle.

For $n\geq 1$ define a rational chain complex $\Pi(n)$ as follows.  An ordered partition of $\{1 \cdc n\}$ is a partition of this set along with a total order on the set of blocks of said partition.   Define $\Pi(n)_i$ to be the span of the set of ordered partitions of $\{1 \cdc n\}$ having $i$ blocks.  Define a differential $\Pi(n)_i\to \Pi(n)_{i-1}$ by
$$
d(B_1\cdc B_i) = \sum_{j=1}^{i-1} (-1)^{j-1} (B_1\cdc B_j\sqcup B_{j+1}\cdc B_i),
$$
where $(B_1\cdc B_j\sqcup B_{j+1}\cdc B_i)$ denotes the ordered partition formed by combining the elements of adjacent blocks $B_j$ and $B_{j+1}$ to form a single new block, while preserving the remaining blocks and the given order.

The chain complex $\Pi(n)$ is isomorphic, after suitable shift in degree, to the cochains on a permutohedron, and the following is well known.  

\begin{lemma}  $H_\ast(\Pi(n)) = \Sigma^n \mathbb{Q}$
\end{lemma}
	
	

Now define an action of the group $\mathbb{Z}_2$ on $\Pi(n)_i$ by declaring that the generator reverses the order of the blocks: 
$$
(B_1, B_2 \cdc B_{i-1}, B_i)\mapsto (-1)^{\lfloor (i-1)/2 \rfloor} (B_i,  B_{i-1}\cdc B_2, B_1)
$$

Let us check that this $\mathbb{Z}_2$ action is compatible with the differential: 
\begin{lemma} $\tau d = d\tau$
\end{lemma}
\begin{proof} A straight-forward computation: $\tau d(B_1\cdc B_i) = $
$$
\sum_{j=1}^{i-1} (-1)^{j-1} \tau(B_1\cdc B_j\sqcup B_{j+1}\cdc B_i) 
= \sum_{j=1}^{i-1} (-1)^{j-1} (-1)^{\lfloor (i-2)/2 \rfloor} (B_i, B_{i-1}\cdc B_j\sqcup B_{j+1}\cdc B_1)
$$
and likewise
$$
d\tau (B_1\cdc B_i) =  (-1)^{\lfloor (i-1)/2 \rfloor} d(B_i\cdc B_1)
 = (-1)^{\lfloor (i-1)/2 \rfloor} \sum_{j=1}^{i-1} (-1)^{i-j+1} (B_i\cdc B_{j+1}\sqcup B_{j}\cdc B_1)
$$
since $B_j\sqcup B_{j+1} = B_{j+1} \sqcup B_{j}$ it's enough to check the signs.  The comparison follows immediately from
$ \lfloor  \frac{i-1}{2} \rfloor + \lfloor  \frac{i-2}{2} \rfloor = i-2. $
\end{proof}

\begin{corollary}  
	$$H_\ast(\Pi(n)_{\mathbb{Z}_2}) \cong \begin{cases}
	\Sigma^n sgn_n& \text{ if } n \text{ is odd}, \\
	                0 & \text{ if } n \text{ is even}.
	\end{cases}
$$
\end{corollary}
\begin{proof}   
	Since we are working with rational coefficients, it suffices to calculate $ H_\ast(\Pi(n))_{\mathbb{Z}_2}$, which can only be non-zero in the top dimension, dimension $n$.  To determine the rank of the homology in this dimension it suffices to compute the Euler characteristic of $\Pi(n)_{\mathbb{Z}_2}$.  Since $\mathbb{Z}_2$ acts freely on $\Pi(n)_i$ for $i\geq 2$ we can calculate the Euler characteristic of $\Pi(n)_{\mathbb{Z}_2}$ in terms of $\Pi(n)$ via:
	$$
	2(\chi(\Pi(n)_{\mathbb{Z}_2})+1) = \chi(\Pi(n))+1 = (-1)^n+1
	$$
	to conclude that $\chi(\Pi(n)_{\mathbb{Z}_2})$ is $0$ if $n$ is even and is $-1$ if $n$ is odd.
	
It remains to confirm the representation type of the non-trivial class when $n$ be odd.  Here $S_n$ acts on $\Pi(n)_n$ simply by relabeling the lone entry in each block, hence $\Pi(n)_n$ is simply the regular representation.  Its decomposition into irreducibles would contain a unique copy of the trivial representation spanned by $\Sigma_{\sigma\in S_n }\sigma(\{1\},\{2\},\{3\}\cdc \{n\})$.  Taking the differential we see the term $(\{1,2\},\{3\}\cdc \{n\})$ appearing exactly twice and with the same sign, hence can't be a cycle.  Thus it must be case that the lone copy of the alternating representation supports the homology. \end{proof}


We now observe that this result concludes the proof.  Each summand in the $\zeta$ filtration is acyclic -- when the order $c$ is odd and $c\geq 3$, the summand is isomorphic to $\Pi(c-1)_{\mathbb{Z}_2}$.  When the order $c$ is even the summand is isomorphic to $\Pi(c-1)_{\mathbb{Z}_2}\to \mathbb{Q}\mathsf{t}_{\vec{g}}$, and this connecting homomorphism is an isomorphism since the blue markings are alternating.  Similarly when the order is $1$, the only differential term removes the tadpole, hence is an isomorphism.

\subsection{(Non)Vanishing of $\mathsf{gr}_{11}H_c^\ast(\op{M}_{2,n})$.}

Combining our Theorem $\ref{decompthm}$ with the results of \cite{PW} we find:
\begin{theorem}\label{nonzero}
	For every $n\geq 13$ there exists injections of $S_n$-representations
	$$\bigoplus_{\substack{0\leq j \leq n-12 \\ j \ \equiv \ n \text{ mod 2} }} H_{2j}(F(\mathbb{R}^3,n-1))^{\oplus 2}\tensor sgn_n\hookrightarrow \mathsf{gr}_{11} H_c^{n+3}(\op{M}_{2,n})  $$ 
	and
	$$\bigoplus_{\substack{0\leq j \leq n-13 \\ j \ \not\equiv \ n \text{ mod 2} }} H_{2j}(F(\mathbb{R}^3,n-1))^{\oplus 2}\tensor sgn_n\hookrightarrow \mathsf{gr}_{11} H_c^{n+2}(\op{M}_{2,n})  $$ 
\end{theorem}
This result is the weight 11 analog of the weight 0 result \cite[Theorem 1.1]{GH}.  It has the the following immediate corollary:

\begin{corollary}\label{nonzerocor}
	$\mathsf{gr}_{11} H_c^j(\op{M}_{2,n}) \neq 0$ if and only if $n\geq 13$ and $j \in \{n+2,n+3\}$ or $10\leq n \leq 12$ and $j=n+3$.  In particular, $\mathsf{gr}_{11} H_c^{n+1}(\op{M}_{2,n})=0$ for all $n$.
\end{corollary}
We dedicate the rest of this subsection to the proof of these results.

\subsubsection{Proof of Theorem $\ref{nonzero}$.}
Expand the decomposition of Theorem $\ref{decompthm}$ via the Pieri rule and consider only those terms corresponding to the alternating representation to see that  $\Sigma^{r-2} B(2,n,r) $ contains a summand isomorphic to  $$\bigoplus_{i=1}^{\lfloor (n-r+1)/2 \rfloor} \Sigma^{2i}\Omega(\mathsf{C}_{(1^{r-1+2i})})(n).$$

Applying Example $\ref{WHex}$, 
and accounting for the shifts by $\Sigma^{2i}$ on the RHS and $\Sigma^{r-2}$ on the LHS of the formula in Theorem $\ref{decompthm}$ we thus see that the cohomology of $B(2,n,r)$ contains a copy of 
$$\bigoplus_{i=1}^{\lfloor (n-r+1)/2 \rfloor} H_{2(n-r+1-2i)}(F(\mathbb{R}^3,n-1))\tensor sgn_n \ \  \text{ and } \ \
\bigoplus_{i=1}^{\lfloor (n-r+1)/2 \rfloor} H_{2(n-r-2i)}(F(\mathbb{R}^3,n-1))\tensor sgn_n$$ in degrees
 $n-2r+3$ and $n-2r+2$ respectively.
Reindexing we find:
$$\bigoplus_{\substack{j=0 \\ j\equiv n-r+1 \text{ mod } 2}}^{n-r-1} H_{2j}(F(\mathbb{R}^3,n-1))\tensor sgn_n \ \  \text{ and } \ \
\bigoplus_{\substack{j=0 \\ j\not\equiv n-r+1 \text{ mod } 2}}^{n-r-2} H_{2j}(F(\mathbb{R}^3,n-1))\tensor sgn_n.$$ 
Finally, we specialize to the case $r=11$, we apply the result of \cite{PW} $$H_c^\ast(\op{M}_{2,n}) \cong \Sigma^{22} H^\ast(B(g,n,11)) \tensor \Delta,$$ where $\Delta:= H^{11}(\overline{\op{M}}_{1,11})$, viewed as a rational vector space of dimension 2, to find these summands concentrated in degrees $n+3$ and $n+2$, as desired.

\subsubsection{Proof of Corollary $\ref{nonzerocor}$.}

We note that Theorem $\ref{decompthm}$ implies that $\mathsf{gr}_{11} H_c^j(\op{M}_{2,n}) = 0$ unless $j= n+3$ or $n+2$.  When $n\geq 13$, Theorem $\ref{nonzero}$ implies $\mathsf{gr}_{11} H_c^j(\op{M}_{2,n})$ is not zero for $j=n+2$ and $n+3$.  The cases $n=10$ and $11$ are given in \cite{PW}.  It thus remains to check the case $n=12$.  Using the computations of \cite{GH} we compute the entire cohomology in this case.

We first apply the result of \cite{FW} to conclude that $B(2,12,11)^\ast$ belongs to a sequence of representation stable chain complexes $B(2,n,n-1)^\ast$ which stabilizes sharply at $n=8$, so it's 
sufficient to compute the homology of $B(2,8,7)$.  For this we apply Theorem $\ref{decompthm}$ to find:
$$H^{\ast-5}( B(2,8,7)) = \Phi[8, (2,1^4)] \oplus \Sigma^2(\Phi[8, (2^3,1^2)] +  \Phi[8,(2^2, 1^4)] + \Phi[8, (2,1^{6})] + \Phi[8,(1^{8})]).$$

The homology of the last four terms is immediate since $\Phi[n,\lambda] = V_\lambda$ (in degree 0) when $\lambda$ is a partition of $n$, so it's enough to compute the homology $\Phi[8, (2,1^4)]$.  For this we invoke the computations of \cite{GH}, who computed $\Phi[n, \lambda]$ for $n\leq 10$.  In particular they computed that $\Phi[8,(2,1^4)]$ is concentrated in degree 2 and has $S_8$-representation isomorphic to 
$[6,1,1]+[5,2,1]+[5,1,1,1]+[4,3,1]+2[4,2,1,1]+[3,3,2]+2[3,3,1,1]+[3,2,2,1]+2[3,2,1^3]+[2,2,2,1,1]+[2,2,1^4]$. 
We therefore conclude that $H^{d}(B(2,12,11))$ is non-zero only if $d=-7$, in which case it is isomorphic to
\begin{eqnarray}\label{15calc}
V:=[6,1^6]+[5,2,1^5]+[5,1^7]+[4,3,1^5]+2[4,2,1^6]+[3,3,2,1^4]+2[3,3,1^6]\\+[3,2,2,1^5]+2[3,2,1^7]  +2[2,2,2,1^6]+2[2,2,1^8] + [2,1^{10}] + [1^{12}] \nonumber
\end{eqnarray}
We remark that this is consistent with the Euler characteristic computed in \cite{PW}, but since it is a decomposition in homology we find:

\begin{corollary} If $j\neq 15$ then $\mathsf{gr}_{11}H_c^j(\op{M}_{2,12}) = 0$.  When $j=15$, we have $\mathsf{gr}_{11}H_c^{15}(\op{M}_{2,12}) = V\tensor \Delta$ where $V$ is the $S_{12}$ representation in Equation $\ref{15calc}$. 
\end{corollary}
This finishes the proof of Corollary $\ref{nonzerocor}$. 
We note that the computation of $\mathsf{gr}_{11}H_c^{\ast}(\op{M}_{2,12})$ is of excess $5$ in the parlance of \cite{PW}.  We remark that for the next case, excess 7, representation stability reduces the computation of $\mathsf{gr}_{11}H_c^{j}(\op{M}_{2,13})$ to knowledge of $\Phi[11,\lambda]$ for $\lambda^\ast=(a,b)$. 

\subsection{Relating $\mathsf{gr}_1(\mathsf{FT}_\Gamma)(n)$ to Configuration Space}

To conclude, we prove the middle statement of Theorem $\ref{ftthm}$.  Recall that for each $n$, $\mathsf{gr}_1(\mathsf{FT}_\Gamma)(n)$ is a chain complex spanned by graphs having $\beta(\gamma) = 1$ and $n$ legs, along with a distinguished vertex of genus $1$, carrying a label in $H^\ast(\Gamma_{1,v})$ of non-zero degree.  The differential is expansion of (non-tadpole) edges.

\begin{lemma}\label{tot}  There is an $S_n$-equivariant quasi-isomorphism
	$$ \mathsf{gr}_1(\mathsf{FT}_\Gamma)(n) \sim \ds\bigoplus_{j=1}^{\lfloor n/2 \rfloor} \Sigma^{4j+1}B(2,n,2j+1).$$
\end{lemma}
\begin{proof}  The chain complex $\mathsf{gr}_1(\mathsf{FT}_\Gamma)(n)$ splits over the internal degree of the unique label $H^{\ast}(\Gamma_{1,m})$.  By Lemma $\ref{g1fa}$, the associated summand of $\mathsf{gr}_1(\mathsf{FT}_\Gamma)(n)_d$ is $0$ unless $d$ is even and satisfies $2\leq d \leq n$.  So it's enough to show that, up to a shift in degree, $\mathsf{gr}_1(\mathsf{FT}_\Gamma)(n)_{2j} \sim	B(2,n,2j+1)$ for the indexed $j$.  This follows by homotopy invariance of the Feynman transform, since the respective sides may be viewed as the genus 2 component of the Feynman transform of weakly equivalent modular cooperads taking the commutative operad in genus $0$ and, in genus 1, taking $\widetilde{\mathsf{C}}_{1^{2j+1}}$ on the one hand and its Koszul resolution (Remark $\ref{BFTrmk}$) on the other.

Regarding degrees, recall that our conventions for the Feynman transform give elements on the left hand side degree $E(\gamma)+d$, where-as elements on the right hand side have degree $E(\gamma) - |R|$.  The homotopy equivalence of modular co-operads sends a degree $d$ class to the minimum number of markings which is $d+1$, hence the shift by $2d+1 = 4j+1$.\end{proof}

Combining this Lemma with Theorem $\ref{decompthm}$ we find:

\begin{proposition} There exists an isomorphism of graded $S_n$-modules:
$$H^\ast(\mathsf{gr}_1(\mathsf{FT}_\Gamma))(n)\cong \bigoplus_\lambda  \Sigma^{|\lambda|+2} p_\lambda\Phi[n,\lambda^\ast] $$
where $p_\lambda$ is non-zero only if $\lambda$ is of the form $\lambda=(a,b)$, in which case:
	$$
	p_{(a,b)} = \begin{cases}
		(a-2)/2	&	a\equiv_2 b\equiv_2 0 \text{ and } b=0\\
		(a-b)/2	&	a\equiv_2 b\equiv_2 0 \text{ and } b\neq 0 \\
		(a-b)/2+1	&	a\equiv_2 b \equiv_2 1 \\
		0	&	a \not\equiv_2 b
	\end{cases}
	$$
\end{proposition}
\begin{proof} Theorem $\ref{mainthm}$ tells us
		$$\Sigma^{2j-1}B(2,n,2j+1)  \sim
\Omega(\mathsf{C}_{(2j-1,1)^\ast})(n) \oplus \ds\bigoplus_{i=1}^{\lfloor(n-2j)/2\rfloor}\Sigma^{2i}\Omega(\mathsf{C}_{(2j-1)^\ast}\circ\mathsf{C}_{((2i+1)^\ast} )(n). $$	


Every term in this decomposition is of the form $\mathsf{C}_{(a,b)^\ast}$, so it remains to count the multiplicity in the sum in Lemma $\ref{tot}$.  Let's first treat the case $b\geq 2$.  If we form $(a,b)$ via the Pieri rule from a term in $\mathsf{C}_{2j-1}\circ \mathsf{C}_{2i+1}$ by adding $k$ boxes to the row $2j-1$, then the possible values for $k$ must satisfy $a=2j-1+k $, $k+b-1=2i$ and $k \leq a-b$. 
Any $k$ with parity opposite $a$ and $b$ will specify a unique $i$ and $j$, hence we find $(a-b)/2$ possibilities when $a$ and $b$ are even and $(a-b)/2+1$ when $a$ and $b$ are odd.

When $b=0$ the calculation is similar, except the $k=1$ solution would force $i=0$ which is not valid.  
Finally, when $b=1$, the solution corresponding to $a=k-1$ would force $j=0$ which is not valid.  However there is an extra term $(2j-1,1)$ in the statement of the result, so the multiplicity remains $(a-b)/2+1$. \end{proof}


It thus remains to compare the multiplicity $p_\lambda$ above to the appropriate summand of $H^\ast_c(F(S^1\vee S^1,-))$.   Observe that the $D_{12}$ representation in statement (2) of Theorem $\ref{ftthm}$ is induced from the 1-dimensional representation $V_2\boxtimes V_{1,1}$ of the group $S_2\times S_2 = \langle X^3\rangle \oplus \langle Y\rangle$.  It therefore remains to verify:

\begin{corollary}  For each $n$ there is an isomorphism of graded $S_n$-representations
	$$
 H^{\ast+2}(\mathsf{gr}_1(\mathsf{FT}_\Gamma))\oplus \mathsf{W}_{\pm} \cong  (H^\ast_c(F(S^1\vee S^1, -))\tensor (V_2\boxtimes V_{1,1}))_{S_2\times S_2}$$ 
\end{corollary}
\begin{proof}  As in \cite[Section 6.1]{GH}, we need only compute the inner product 
	$$
	p^\prime_\lambda = \langle Res(\mathbb{S}_{\lambda}(\mathbb{Q}^2)),V_2\boxtimes V_{1,1}\rangle
	$$
	along the restriction $\langle X^3\rangle \oplus \langle Y\rangle \hookrightarrow GL_2(\mathbb{Z})$.  Carrying out this computation as above (see Lemma $\ref{qlambda}$) we have non-zero contributions coming from $\lambda = (a,b)$, in which case we compute:
			$$ 	p^\prime_{(a,b)} = \begin{cases}
		(a-b)/2	&	a\equiv_2 b\equiv_2 0 \\
		(a-b)/2+1	&	a\equiv_2 b \equiv_2 1 \\
		0	&	a \not\equiv_2 b
	\end{cases}
	$$
	Comparing to the formula for $p$ we find agreement with $p^\prime$ except when $(a,b)=(2i,0)$.  Adding the $\mathsf{W}_\pm$ term makes the correction.\end{proof}


%

{\bf Acknowledgment:} We would like to thank 
Nir Gadish, 
Louis Hainaut,
Sam Payne,
Dan Petersen,
Geoffrey Powell and
Thomas Willwacher 
for conversations and/or email correspondence which helped shape our understanding of the ideas we've looked to tie together in this article. BW gratefully acknowledges support from Simons Foundation Grant No.\ 704658.

\end{document}